\documentclass[12pt,reqno]{article}
\usepackage[utf8]{inputenc}
\usepackage{amsmath}
\usepackage{amsfonts}
\usepackage{amssymb}
\usepackage{amsthm}	
\usepackage{cases}
\usepackage{subeqnarray}
\usepackage{cite}
\usepackage[colorlinks]{hyperref}
\usepackage{graphics}
\usepackage{graphicx}
\usepackage{fancyref}
\usepackage{epstopdf}
\usepackage{float}
\usepackage{enumerate} 
\usepackage{bm}
\usepackage{titlesec}
\usepackage[titletoc,toc,title]{appendix}
\usepackage{mathrsfs}

\makeatletter
\allowdisplaybreaks \numberwithin{equation}{section}
\newtheorem{theorem}{Theorem}[section]

\newtheorem{lemma}{Lemma}[section]

\newtheorem{definition}{Definition}[section]
\newtheorem{remark}{Remark}[section]

\newtheorem{thmx}{Theorem}

\begin{document}

\title{Asymptotic behavior of the principal eigenvalue and basic reproduction ratio for periodic patch models
\date{\empty}
\author{ Lei Zhang $^{a,b}$, Xiao-Qiang Zhao $^{b}$	\\
	{\small a Department of Mathematics, Harbin Institute of Technology at Weihai,}\\
	{\small Weihai, Shandong 264209, China.}\\
	 {\small b Department of Mathematics and Statistics, Memorial University of Newfoundland,}\\
	 {\small St. John's, NL A1C 5S7, Canada.}\\
}
}\maketitle
\vspace{-5mm}
 \centerline{(SCIENCE CHINA Mathematics, to appear)}
 
\begin{abstract}
	This paper is devoted to the study of the asymptotic behavior of the principal eigenvalue and basic reproduction ratio associated with periodic population models in a patchy environment for small and large dispersal rates. We first deal with the eigenspace corresponding to the zero eigenvalue of the connectivity matrix. Then we investigate the limiting profile of the principal eigenvalue of an associated periodic eigenvalue problem as the dispersal rate goes to zero and infinity, respectively. We further establish the asymptotic behavior of the basic reproduction ratio in the case of small and large dispersal rates. Finally, we apply these results to a periodic Ross-Macdonald patch model.
	
\end{abstract}
\section{Introduction}

In 2007, Allen et al. \cite{allen2007Asymptotic} studied the following
epidemic model in a patchy environment:
\begin{equation}\label{equ:SIS:auto}
	\begin{aligned}
		\frac{\mathrm{d} S_i}{\mathrm{d} t}= d_S \sum_{j =1}^{n} l_{ij} S_j - \beta_i \frac{S_i I_i}{S_i + I_i} + \gamma_i I_i, &~ i=1,\cdots,n,\\
		\frac{\mathrm{d} I_i}{\mathrm{d} t}= d_I \sum_{j =1}^{n} l_{ij} I_j + \beta_i \frac{S_i I_i}{S_i + I_i} - \gamma_i I_i, &~ i=1,\cdots,n.\\
	\end{aligned}
\end{equation}
Here $n \geq 2$ is the number of patches,  $S_i(t)$ and $I_i(t)$ are the numbers of susceptible and infected individuals in patch $i$ at time $t$, respectively. The parameters $d_S$ and $d_I$ are the migration rate of susceptible and infected populations; $l_{ij}$ is a nonnegative constant which denotes the degree of movement from patch $j$ to patch $i$ for $j \neq i$ and $l_{ii}=- \sum_{j \neq i}l_{ji}$ is the degree of movement from patch $i$ to all other patches; $\beta_i \geq 0$ and $\gamma_i >0$ are disease transmission and recovery rates at patch $i$, respectively. Let $L=(l_{ij})_{n \times n}$, $F=\mathrm{diag}(\beta_1,\cdots,\beta_n)$ and $V=\mathrm{diag}(\gamma_1,\cdots,\gamma_n)$. Following \cite{diekmann1990definition,van2002reproduction}, the basic reproduction ratio of system \eqref{equ:SIS:auto} is expressed as  $\mathcal{R}_0(d_I)=r((V-d_I L)^{-1}F)$, $d_I \geq 0$, where  $r((V-d_I L)^{-1}F)$ is the spectral radius of $(V-d_I L)^{-1}F$.  

Recall that  a square matrix is said to be cooperative if its off-diagonal elements are nonnegative, and nonnegative if all elements are nonnegative;
a square matrix is said to be irreducible if it is not similar, via a permutation, to a block lower triangular matrix, and reducible if otherwise;
and the spectral bound (also called the stability modulus) of  a square matrix $A$ is defined as 
$s(A)=\sup \{\mathrm{Re} \lambda: \lambda \text{ is an eigenvalue of } A\}$.

Under the assumption that the migration matrix $L$ of infected individuals is symmetric and irreducible, Allen et al. \cite{allen2007Asymptotic} showed that
$$
\lim\limits_{d_I \rightarrow 0^+} s(d_I L -V+F) =\max_{1 \leq i \leq n}(\beta_i -\gamma_i), ~\lim\limits_{d_I \rightarrow +\infty} s(d_I L -V+F)= \frac{1}{n}\sum_{i =1}^{n} (\beta_i -\gamma_i),
$$
$$
\lim\limits_{d_I \rightarrow 0^+}\mathcal{R}_0(d_I)= \mathcal{R}_0(0)=\max_{1 \leq i \leq n} \frac{\beta_i}{\gamma_i}, \text{ and }
\lim\limits_{d_I \rightarrow +\infty}\mathcal{R}_0(d_I)= \frac{\sum_{i =1}^{n}\beta_i}{\sum_{i =1}^{n} \gamma_i}.
$$
Without assuming the symmetry of $L$, Gao and Dong \cite{gao2019travel,gao2020fast} and Chen et al. \cite{chen2020asymptoticJMB}  recently proved the same limiting properties  for $s(d_I L -V+F)$ and $\mathcal{R}_0(d_I)$ as  $d_I \rightarrow 0^+$, and generalized the other two limits into
$$
\lim\limits_{d_I \rightarrow +\infty} s(d_I L -V+F)= \sum_{i =1}^{n} (\beta_i -\gamma_i) q_i, \quad
\lim\limits_{d_I \rightarrow +\infty}\mathcal{R}_0(d_I)= \frac{\sum_{i =1}^{n}\beta_i q_i}{\sum_{i =1}^{n} \gamma_i q_i}, 
$$
where  $\bm{q}=(q_1,\cdots,q_n)^T$ is a right eigenvector of $L$ corresponding to  the eigenvalue $0$ such that  $\sum_{i=1}^{n} q_i=1$.

Note that the connectivity matrix  obtained from the linearization 
of system \eqref{equ:SIS:auto} at the disease-free equilibrium refers to the migration matrix of infected individuals.
In many multi-population models in a patchy environment, however, the connectivity matrix is reducible, although the migration matrix for each population is irreducible (see, e.g., \cite{gao2014periodic,gao2012multipatch}).
Thus, a natural question is how to  further characterize the above limiting profiles 
for  $s(d_I L -V+F)$ and $\mathcal{R}_0(d_I)$ without the irreducibility condition on the connectivity matrix. 
Such problems have been explored for reaction-diffusion systems (see, e.g., \cite{allen2008asymptotic,wang2012basic,magal2019basic,chen2020asymptoticSIAP,dancer2009principal,lam2016asymptotic,zhang2020asymptotic}).
In the case where the connectivity matrix is symmetric, this question is much easier than the associated problem for reaction-diffusion systems. 
It is worthy pointing out that the limiting problem for large dispersal rate
is highly nontrivial when the connectivity matrix is non-symmetric.

For time-periodic patch population models (see, e.g., \cite{gao2014periodic,zhang2007periodic}),  we may conjecture that 
similar  limiting results on the principal eigenvalue 
and basic reproduction ratio hold true.  This conjecture was 
confirmed for reaction-diffusion systems (see, e.g., \cite{hutson2001evolution,yang2019dynamics,zhang2020asymptotic,peng2012reaction,peng2015effects}).  However, it seems that these methods and arguments  may not be well adapted to such periodic patch models due to the  lack of irreducibility and symmetry for the connectivity matrix.

Our purpose of this paper is to address the afore-mentioned two
questions for patch population models.  Motivated by \cite{wang2004epidemic,zhang2007periodic,allen2007Asymptotic,gao2012multipatch,gao2020fast}, we assume that  the connectivity matrix $L$ admits the property that 
\begin{itemize}
	\item[(H1)] $L=(l_{ij})_{n \times n} $ is an $n\times n$ cooperative matrix with zero column sums.
\end{itemize}
Then we have the following elementary observation, which plays a key role in our analysis.

\begin{thmx}[see Lemmas \ref{lem:L_K}, \ref{lem:L_K_M} and \ref{lem:PQM}] \label{thm:A}
	Assume that {\rm (H1)} holds. Let $\alpha_0$ be the algebraic multiplicity of the zero eigenvalue of $L$.
	Then the following statements are valid:
	\begin{itemize}
		\item[\rm (i)]There exist nonnegative matrices $P=(p_{hj})_{\alpha_0 \times n}$ and $Q=(q_{il})_{n \times \alpha_0}$ such that $PL$ and $LQ$ are zero matrices and $PQ$ is an $\alpha_0 \times \alpha_0$ identity matrix.
		\item[\rm (ii)] If $M$ is an $n \times n$ cooperative matrix, then 
		$PMQ$ is an $\alpha_0 \times \alpha_0$ cooperative matrix.
		\item[\rm (iii)] Let $\hat{P}=(\hat{p}_{hj})_{\alpha_0 \times n}$ and $\hat{Q}=(\hat{q}_{il})_{n \times \alpha_0}$ be two nonnegative matrices such that $\hat{P}L$ and $L\hat{Q}$ are zero matrices and $\hat{P}\hat{Q}$ is an $\alpha_0 \times \alpha_0$ identity matrix. Then $P M Q$ is similar to $\hat{P} M \hat{Q}$.
	\end{itemize}
\end{thmx}

We remark that all rows of $P$ and columns of $Q$ are the left and right eigenvectors of $L$, respectively.  Note that 
any autonomous system can be regarded as a periodic one
with the period being any given positive number. As a
straightforward consequence of our general result for periodic
systems (see Theorem \ref{thm:C} below), we have the following result on the limiting profiles of the spectral bound and basic reproduction ratio with small and large dispersal rate for autonomous patch models.

\begin{thmx}[]\label{thm:B}
	Assume that  {\rm (H1)} holds, $-V$ is an $n \times n $ cooperative matrix, and $F$ is an $n \times n $ nonnegative matrix. Let $P$ and $Q$ 
	be defined as in Theorem \ref{thm:A},  $\tilde{V}:= PVQ$ and $\tilde{F}:=PFQ$. Then the following statements are valid:
	\begin{itemize}
		\item[\rm (i)] $\lim\limits_{d \rightarrow 0^+}  s(d L -V +F)= 
		s(-V+F)$ and 
		$\lim\limits_{d \rightarrow +\infty} s(d L -V +F)= s(-\tilde{V}+\tilde{F})$.
		\item[\rm (ii)]  If, in addition, $s(dL-V)<0$ for all $d \geq 0$ and $s(\tilde{V})<0$, then $\lim\limits_{d \rightarrow 0^+} \mathcal{R}_0 (d)  =\mathcal{R}_0(0)$ and $\lim\limits_{d \rightarrow +\infty}  \mathcal{R}_0= \tilde{\mathcal{R}}_0$, where $\mathcal{R}_0(d):=r((V-dL)^{-1}F)$, $\forall d \geq 0$, and $\tilde{\mathcal{R}}_0:= r (\tilde{V}^{-1} \tilde{F})$.
	\end{itemize}
\end{thmx}

Note that  the additional conditions $s(dL-V)<0$ for all $d \geq 0$ and $s(\tilde{V})<0$ are used to guarantee that the associated basic reproduction ratios $\mathcal{R}_0(d)$ and $\tilde{\mathcal{R}}_0$ are well defined (see, e.g., \cite{van2002reproduction}), and 
$s(-\tilde{V}+\tilde{F})$ is independent of the choice of $P$ and $Q$ due to Theorem \ref{thm:A}. In the case where  $L$ is irreducible, the results in Theorem \ref{thm:B} were established in \cite{gao2019travel,gao2020fast,chen2020asymptoticJMB}.

To present our main result for time-periodic systems, we use $T>0$ to denote the period throughout this paper. Let $F(t)$ and 
$V(t)$ be two continuous  $n\times n$ matrix-valued functions of $t \in \mathbb{R}$ such that 
\begin{itemize}
	\item[(H2)] $F(t+T)=F(t)$, $V(t+T)=V(t)$,  $F(t)$ is nonnegative, and $-V(t)$ is cooperative for all $t \in \mathbb{R}$.
\end{itemize}

For any $t \in \mathbb{R}$,  let $\tilde{F}(t):=PF(t)Q$ and $\tilde{V}(t):=PV(t)Q$, where $P$ and $Q$ are defined as in Theorem \ref{thm:A}. For any $d \geq 0$, let $\{\Phi_{d}(t,s): t \geq s\}$ be the evolution family on $\mathbb{R}^n$ of 
$
\frac{\mathrm{d} \bm{v}}{\mathrm{d} t}=d L \bm{v} -V(t) \bm{v},
$ 
and let $\{\tilde{\Phi}(t,s): t \geq s\}$ be the evolution family on $\mathbb{R}^{\alpha_0}$ of 
$\frac{\mathrm{d} \bm{v}}{\mathrm{d} t}= - \tilde{V}(t) \bm{v}$ (see Definition \ref{def:evol}), where $\alpha_0$ is the algebraic multiplicity of the zero eigenvalue of $L$. 
Let $\omega(\Phi)$ be the exponential growth bound of an evolution family $\Phi$ (see Definition \ref{def:evol}). We further assume that 
\begin{itemize}
	\item[(H3)]  $\omega(\Phi_{d})<0$ for all $d \geq 0$ and $\omega(\tilde{\Phi})<0$.
\end{itemize}

For any $d \geq 0$, let $\lambda_{d}^{*}$ be the principal eigenvalue of
the periodic eigenvalue problem (see Definition \ref{def:principal} and Theorem \ref{thm:existence}): 
$$\frac{\mathrm{d} \bm{u}}{\mathrm{d} t}=d L \bm{u} -V(t) \bm{u} +F(t) \bm{u} - \lambda \bm{u}.$$
According to \cite{bacaer2006epidemic, wang2008threshold}, 
the basic reproduction ratio $\mathcal{R}_0(d)$ is well defined for 
the following periodic ODE system (see section \ref{sec:R0}):
\begin{equation}\label{equ:dLVF}
	\frac{\mathrm{d} \bm{v}}{\mathrm{d} t}=d L \bm{v} -V(t) \bm{v} +F(t) \bm{v}.
\end{equation}
In view of Theorem \ref{thm:A}, we see that  $-\tilde{V}(t)$ is cooperative for any $t \in \mathbb{R}$. Moreover, $\tilde{F}(t)$ is nonnegative  for any $t \in \mathbb{R}$.
Let $\tilde{\lambda}^{*}$ be the principal eigenvalue of  
the periodic eigenvalue problem:
$$\frac{\mathrm{d} \bm{u}}{\mathrm{d} t}= -\tilde{V}(t) \bm{u} + \tilde{F}(t) \bm{u} - \lambda \bm{u},$$
and $\tilde{\mathcal{R}}_0$ be the basic reproduction ratio of the following periodic equation (see section \ref{sec:R0}):
\begin{equation}\label{equ:tdLVF} 
	\frac{\mathrm{d} \bm{v}}{\mathrm{d} t}=-\tilde{V}(t) \bm{v} + \tilde{F}(t) \bm{v}.
\end{equation}
Then we have the following result on the asymptotic behavior of 
$\lambda_{d}^{*}$ and $\mathcal{R}_0(d)$ for periodic patch models.

\begin{thmx}[see Theorems \ref{thm:eig} and \ref{thm:R0}] \label{thm:C}
	Assume that {\rm (H1)--(H3)} hold.
	Then the  following statements are valid:
	\begin{itemize}
		\item[\rm (i)] $\lim\limits_{d \rightarrow 0^+} \lambda_{d}^{*}= 
		\lambda_{0}^{*}$ and 
		$\lim\limits_{d \rightarrow +\infty} \lambda_{d}^{*}= \tilde{\lambda}^{*}$.
		\item[\rm (ii)] $\lim\limits_{d \rightarrow 0^+}  \mathcal{R}_0(d)= \mathcal{R}_0(0)$ and $\lim\limits_{d \rightarrow +\infty}  \mathcal{R}_0(d)= \tilde{\mathcal{R}}_0$.
	\end{itemize}
\end{thmx}

We should point out that $\tilde{\lambda}^{*}$ is independent of the choice of $P$ and $Q$ (see Lemma \ref{lem:PQO}). The statements
(i) and (ii) in	Theorem \ref{thm:C} are straightforward consequences of Theorems \ref{thm:eig} and \ref{thm:R0}, respectively. 
In Theorem \ref{thm:R0},  we also introduce a metric space of parameters to discuss the continuity of the basic reproduction ratio with respect to parameters.

Since  the Poincar\'e (period)  map of system \eqref{equ:dLVF}, which is a square matrix, is continuous with respect to the dispersal rate $d \in[0,+\infty)$, so is the principal eigenvalue due to the standard matrix perturbation theory.  To obtain the limiting profile of the principal eigenvalue as the dispersal rate goes to infinity,  we distinguish two cases.
In the case where the Poincar\'e  map (matrix) of \eqref{equ:tdLVF}  is irreducible,  we use some ideas inspired  by \cite{hale1986large,hale1987varying,hale1989shadow,hutson2001evolution,zhang2020asymptotic},  where the asymptotic behavior of the positive steady states or periodic solutions was derived for large diffusion coefficients. In the case where  such a matrix is reducible, we combine  the perturbation technique and  the results for appropriate subsystems such that the  Poincar\'e  maps of the associated limiting systems are irreducible.  In our recent paper \cite{zhang2020asymptotic}, we established the continuity of the basic reproduction ratio with respect to parameters under the setting of Thieme \cite{thieme2009spectral}, which enables us to reduce the  limiting profile of the basic reproduction ratio to the asymptotic behavior of the principal eigenvalue of the
associated periodic eigenvalue problem with parameters. In the current paper, we give a more general result in this regard and then use  it to prove Theorem \ref{thm:C} (ii).

The remaining part of this paper is organized as follows. 
In the next section, we present some basic properties of cooperative matrices and prove a general result in order to study the continuity of the basic reproduction ratio with respect to parameters. In section \ref{sec:eig}, we study the asymptotic behavior of the principal eigenvalue for periodic cooperative ODE systems with large dispersal rate. In section \ref{sec:R0}, we prove the continuity of the basic reproduction ratio with respect to the dispersal rate and investigate the limiting profile of the basic reproduction ratio as the dispersal rate goes to infinity.
As an  illustrative example, we also
apply these analytic results to a periodic Ross-Macdonald patch model.

\section{Preliminaries}

In this section, we present some properties of cooperative matrices and prove a general result in order to study the continuity of the basic reproduction ratio with respect to parameters.
Throughout the whole paper, we denote $\bm{0}=(0,\cdots,0)^T$ in the case of any finite dimension. Moreover, without ambiguity, $0$ refers to zero matrix. 
\begin{lemma}\label{lem:L_K:0}
	Assume that {\rm (H1)} holds and $L$ can be split into a block lower triangular matrix 
	$$
	L=
	\left(
	\begin{array}{ccc}
		L_{11}& \cdots & L_{1\alpha} \\ 
		\vdots& \ddots& \vdots\\
		L_{\alpha 1}& \cdots & L_{\alpha \alpha}\\
	\end{array} 
	\right)
	$$
	such that $L_{hh}$ is an $n_h \times n_h $ irreducible matrix for $1 \leq h  \leq \alpha$ with $\sum_{l=1}^{\alpha} n_l=n$, and $L_{hl}=0$ for $1 \leq h < l \leq \alpha$.  Then for any fixed $1 \leq l \leq \alpha$,  
	$s(L_{ll})=0$ if $L_{hl}=0$ for all $1 \leq l < h \leq \alpha$,  and $s(L_{ll})<0$ if otherwise.  Equivalently, $L_{hl}=0$  for all $1 \leq h \neq l \leq \alpha$ if $s(L_{ll})=0$, and there is some $h_0 \neq l$ such that $L_{h_0l}$ is a nonzero matrix if otherwise.
\end{lemma}

\begin{proof}
	Let $\bm{e}=(1,\cdots,1)^T$ and $\bm{e}_l=(1,\cdots,1)^T$ be $n$ and $n_l$-dimensional vectors for any $1 \leq  l \leq \alpha$, respectively. It is easy to see that $\bm{e}^T L= \bm{0}^T$. For a fixed $1 \leq l_0 \leq \alpha$, an easy computation yields  that 
	$$
	\sum_{h =1}^{\alpha}(\bm{e}_h)^T L_{h l_0}=\sum_{h =l_0}^{\alpha}(\bm{e}_h)^T L_{h l_0} =\bm{0}^T.
	$$  
	If $L_{hl_0}$ is a zero matrix for all $1 \leq l_0 < h \leq \alpha$, then $(\bm{e}_{l_0})^T L_{l_0l_0} = \bm{0}^T$. This implies that $s(L_{l_0l_0})=0$.
	If otherwise, by the irreducibility of $L_{l_0l_0}$, we conclude that $s(L_{l_0l_0})<0$ due to \cite[Theorem II.1.11]{berman1994nonnegative}.
\end{proof}

\begin{itemize}
	\item[(H1)$'$] $L=(l_{ij})_{n \times n} $ is an $n\times n$ cooperative matrix with $s(L)=0$, and $L$ can be split into a block lower triangular matrix 
	$$
	L=
	\left(
	\begin{array}{ccc}
		L_{11}& \cdots & L_{1\alpha} \\ 
		\vdots& \ddots& \vdots\\
		L_{\alpha 1}& \cdots & L_{\alpha \alpha}\\
	\end{array} 
	\right)
	$$
	such that  $L_{hh}$ is an $n_h \times n_h $ irreducible matrix for $1 \leq h  \leq \alpha$  with $\sum_{l=1}^{\alpha} n_l=n$, $L_{hl}=0$ for $1 \leq h < l \leq \alpha$, and $L_{hl}=0$  for all $l \in \Lambda_0$ and $1\leq h \leq \alpha$ with $h \neq l$, where 
	$$\Lambda_0:=\{1 \leq l \leq \alpha: s(L_{ll})=0 \}, \text{ and } \Lambda_0^{c}:=\{1 \leq l \leq \alpha: s(L_{ll})<0 \}.$$ 
	Let $\alpha_0$ and $\alpha_0^c$ denote the number of all elements in $\Lambda_0$ and $\Lambda_0^{c}$, respectively.
\end{itemize}

In the use of Lemma  \ref{lem:L_K:0},  we choose $\alpha=1$ if
$L$ is irreducible, and write $L$ 
as such a block lower triangular matrix  via a permutation  if $L$ is reducible. Accordingly, Lemma \ref{lem:L_K:0} implies that 
(H1) is sufficient for (H1)$'$ to hold.

\begin{lemma}\label{lem:L_K}
	Assume that {\rm (H1)$'$} holds. Let $\bm{\nu}$ be an $\alpha_0$-dimensional vector defined by
	$
	\bm{\nu}=(\nu_1,\cdots,\nu_{\alpha_0})^T
	$
	with $\nu_l= \alpha_0^c +l$, $\forall 1 \leq l \leq \alpha_0$. Then the following statements are valid:
	\begin{itemize}
		\item[\rm (i)] 
		If $\Lambda_0^c \neq \emptyset$, then $\Lambda_0^c=\{1,\cdots,\alpha_0^c\}$ and $\Lambda_0=\{\alpha_0^c+1,\cdots,\alpha\}$, via a permutation. 
		\item[\rm (ii)] The algebraic multiplicity of the zero eigenvalue of $L$ is $\alpha_0$, and there exist $\alpha_0$ linearly independent left positive eigenvectors $(\bm{p}_l)^T:=((\bm{p}_l^1)^T,\cdots,(\bm{p}_l^\alpha)^T)$, $1 \leq l \leq \alpha_0$ of $L$ and right positive eigenvectors $\bm{q}_l=((\bm{q}_l^1)^T,\cdots,(\bm{q}_l^\alpha)^T)^T$, $1 \leq l \leq \alpha_0$ of $L$ corresponding to $0$ such that $(\bm{p}_l )^T\bm{q}_h=\delta_{lh}$ for $1 \leq l, h \leq \alpha_0$, where $\bm{p}_l^i$ and $\bm{q}_l^i$ are $n_i$-dimensional vectors and $\delta_{lh}$ denotes the Kronecker delta function(that is, $\delta_{lh}=1$ if $l=h$ and $\delta_{lh}=0$ if otherwise). Moreover,  $\bm{p}_l^{\nu_l} \gg \bm{0}$, $\bm{q}_l^{\nu_l} \gg \bm{0}$, $L_{\nu_l\nu_l}\bm{q}_l^{\nu_l}=\bm{0}$ and $(\bm{p}_{l}^{\nu_l})^TL_{\nu_l\nu_l}=\bm{0}^T$, $\forall 1 \leq l \leq \alpha_0$.
		\item[\rm (iii)] Define
		$$		
		P:=
		\left(\begin{matrix}
			\bm{p}_1,
			\cdots,
			\bm{p}_{\alpha_0}
		\end{matrix}
		\right)^T
		\text{ and }
		Q:=
		\left(\begin{matrix}
			\bm{q}_1,
			\cdots,
			\bm{q}_{\alpha_0}
		\end{matrix}
		\right).
		$$
		Then $PQ$ is an $\alpha_0 \times \alpha_0$ identity matrix. Moreover, $PL=0$ and $LQ=0$.
	\end{itemize}
	
\end{lemma}

\begin{proof}
	(i) can be derived by a permutation due to (H1)$'$.
	
	(ii) We only consider the case of $\alpha_0^c >0$, since the case of $\alpha_0^c =0$ can be obtained similarly.
	For any $1 \leq l \leq \alpha_0$, choose $\bm{q}^{\nu_l} \gg 0$ such that $L_{\nu_l\nu_l}\bm{q}^{\nu_l}=\bm{0}$, and define 
	$$
	\bm{q}_l:=((\bm{q}_l^1)^T,\cdots,(\bm{q}_l^\alpha)^T)^T,
	$$ where $\bm{q}_l^i$ is an $n_i$-dimensional vector,  $\bm{q}_l^{\nu_l}=\bm{q}^{\nu_l}$, and $\bm{q}_l^i=\bm{0}$ if $i \neq \nu_l$.  This implies that $L\bm{q}_l=0$ for any $1 \leq l \leq  \alpha_0$.
	
	For any $1 \leq l \leq \alpha_0$,  choose $\bm{p}^{\nu_l} \gg 0$ such that $(\bm{p}^{\nu_l})^T L_{\nu_l\nu_l}=\bm{0}^T$  with $(\bm{p}^{\nu_l})^T\bm{q}^{\nu_l}=1$.
	Define 
	$$
	\bm{p}_l:=((\bm{p}_l^1)^T,\cdots,(\bm{p}_l^\alpha)^T)^T,
	$$ where $\bm{p}_l^i$ is an $n_i$-dimensional vector,  $\bm{p}_l^{\nu_l}=\bm{p}^{\nu_l}$, and $\bm{p}_l^i=\bm{0}$ if $i \in \Lambda_0$ with $i \neq \nu_l$, and $\bm{p}_l^i$ is solved by the following equations if $i \in \Lambda_0^c$.
	\begin{equation}\label{equ:p:Lhl}
		\sum_{i =1}^\alpha (\bm{p}_l^i)^T L_{ih}=\bm{0}^T, ~1 \leq  h  \leq \alpha. 
	\end{equation}
	This is equivalent to 
	$$
	\sum_{i \in \Lambda_0^c} (\bm{p}_l^i)^T L_{ih}=-\sum_{i \in \Lambda_0} (\bm{p}_l^i)^T L_{ih},~ 1 \leq  h  \leq \alpha. 
	$$
	Notice that the coefficient matrix $$\tilde{L}=\left(
	\begin{array}{ccc}
		L_{11}& \cdots & L_{1\alpha_0^c} \\ 
		\vdots& \ddots& \vdots\\
		L_{\alpha_0^c1}& \cdots & L_{\alpha_0^c\alpha_0^c}\\
	\end{array} 
	\right)$$ of the above equations is a block lower triangular matrix whose diagonal elements are $L_{ii}$ for $i \in \Lambda_0^c$. It then follows that $s(\tilde{L})= \max_{i \in \Lambda_0^c}s(L_{ii})<0$, and hence, \eqref{equ:p:Lhl} adimits a unique solution. Moreover, $\bm{p}_l^i$ is nonnegative for $i \in \Lambda_0^c$ since $\tilde{L}$ is cooperative. Thus, $\bm{p}_l^T L =\bm{0}^T$, $\forall 1 \leq  l \leq \alpha_0$.	
	
	In view of the above arguments, it easily follows that the algebraic multiplicity of the zero eigenvalue of $L$ is no less than $\alpha_0$.	To obtain the converse statement, it suffices to prove the following two claims.
	
	{\it Claim 1.} If $L\bm{q}=\bm{0}$ with $\bm{q}=((\bm{q}^1)^T,\cdots,(\bm{q}^\alpha)^T)^T$, where $\bm{q}^l$ is an $n_l$-dimensional vector, then $\bm{q}^{l}=\bm{0}$ for $l \in \Lambda_0^c$ and $L_{ll}\bm{q}^l=\bm{0}$ for $l \in \Lambda_0$.
	
	{\it Claim 2.} If $L^m\bm{q}=\bm{0}$ for some $m>1$ with $\bm{q}=((\bm{q}^1)^T,\cdots,(\bm{q}^\alpha)^T)^T$, where $\bm{q}^l$ is an $n_l$-dimensional vector, then $\bm{q}^{l}=\bm{0}$ for $l \in \Lambda_0^c$ and $L_{ll}\bm{q}^l=\bm{0}$ for $l \in \Lambda_0$.
	
	Let us postpone the proof of these claims, and complete the proof in a few lines.  
	By the irreducibility of $L_{ll}$, it then follows from Claim 1 that $\bm{q}$ is a linear combination of $\{\bm{q}_l \}_{1 \leq l \leq \alpha }$ for any $\bm{q}$ with $L\bm{q}=\bm{0}$, and hence, the geometric multiplicity of the zero eigenvalue of $L$ is no more than $\alpha_0$. Similarly, it follows from Claim 2 that the algebraic multiplicity of the zero eigenvalue of $L$ is no more than $\alpha_0$. Thus, the desired conclusion holds. 
	
	We now return to the proof of Claim 1, and first show that $\bm{q}^{l}=\bm{0}$ for all $l \in \Lambda_0^c$ by the induction method.
	It is easy to see that $L_{11}\bm{q}^1=\bm{0}$. Thus, $s(L_{11})<0$ implies that $\bm{q}^1=\bm{0}$.  Assume that $\bm{q}^{l}=\bm{0}$ for $1 \leq l \leq l_0$ with ${l_0} \in \Lambda_0^c$, it suffices to prove $\bm{q}^{{l_0}+1}=\bm{0}$ if $l_0 + 1 \in \Lambda_0^c$. In view of $L\bm{q}=\bm{0}$, we have
	$$
	L_{({l_0}+1)({l_0}+1)}\bm{q}^{{l_0}+1}
	=\sum_{i=1}^{\alpha} L_{({l_0}+1) i} \bm{q}^{i}
	=\bm{0},
	$$
	due to $L_{({l_0}+1) i }=0$ if $i > l_0 +1$, and $\bm{q}^{i}=\bm{0}$ if $i < l_0 +1$. Thus, $\bm{q}^{l}=\bm{0}$ for all $l \in \Lambda_0^c$. By (H1)$'$, $L_{hl}=0$ if $ h,l \in \Lambda_0$ with $h \neq l$.
	It then follows that $$L_{ll}\bm{q}^l=\sum_{i =1}^{\alpha}L_{li}\bm{q}^i=\bm{0},~l \in \Lambda_0.$$
	
	We next verify Claim 2. Since $L$ is a block lower triangular matrix, $\mathrm{diag}(L_{11}^m,\cdots,L_{\alpha \alpha}^m)$ is the block diagonal of $L^m$. 
	In view of  Claim 1, we have $\bm{q}^{l}=\bm{0}$ for $l \in \Lambda_0^c$ and $L_{ll}^{m}\bm{q}^l=\bm{0}$ for $l \in \Lambda_0$.
	Thus, the irreducibility of $L_{ll}$ implies that $L_{ll}\bm{q}^l=\bm{0}$.
\end{proof}

In view of Lemma \ref{lem:L_K}, we observe  that $\alpha_0$ is 
not only the number of the elements in $\Lambda_0$, but also the algebraic multiplicity of the zero eigenvalue of $L$. In the rest of 
this paper, we use the same notations $\bm{\nu}$, $\bm{p}_l$, $\bm{q}_l$, $P$ and $Q$ as in Lemma \ref{lem:L_K}.

\begin{lemma}\label{lem:L_K_M}
	Assume that {\rm (H1)$'$} holds,  $\Lambda_0^c=\{1,\cdots,\alpha_0^c\}$ and $\Lambda_0=\{\alpha_0^c+1,\cdots,\alpha\}$ whenever $\Lambda_0^c \neq \emptyset$. Let $M$ be a cooperative matrix such that
	$$
	M=
	\left(
	\begin{array}{ccc}
		M_{11}& \cdots & M_{1\alpha} \\ 
		\vdots& \ddots& \vdots\\
		M_{\alpha 1}& \cdots & M_{\alpha \alpha}\\
	\end{array} 
	\right),
	$$
	where $M_{hl}$ is an $n_h \times n_l $ matrix for $1 \leq h,l  \leq \alpha$. Then the following statements are valid:
	\begin{itemize}
		\item[\rm (i)] $PMQ$ is cooperative.
		\item[\rm(ii)] Let $\bm{b}$ be an $\alpha_0$-dimensional vector defined by
		$
		\bm{b}=(b_1,\cdots,b_{\alpha_0})^T
		$
		with $1 \leq b_i \leq \alpha_0$ and $b_i \neq b_j$ if $i \neq j$.
		Define a matrix $\tilde{M}:=(\tilde{m}_{hl})_{\alpha_0 \times \alpha_0}$ by
		$\tilde{m}_{hl}=\bm{p}_{b_h}^T M \bm{q}_{b_l}$.  Then $\tilde{M}$ is similar to $PMQ$, via a permutation. 
		If $\tilde{M}$ is reducible, then $\tilde{M}$, via exchanging the order of the components of $\bm{b}$, can be split into
		$$
		\tilde{M}=
		\left(
		\begin{array}{ccc}
			\tilde{M}_{11}& \cdots & \tilde{M}_{1\tilde{n}} \\ 
			\vdots& \ddots& \vdots\\
			\tilde{M}_{\tilde{n}1}& \cdots & \tilde{M}_{\tilde{n}\tilde{n}}\\
		\end{array} 
		\right),
		$$
		where $\tilde{M}_{ii}$ is an $\alpha_i \times \alpha_i$ irreducible matrix for all $1 \leq i \leq \tilde{n}$ with $\sum_{i=1}^{\tilde{n}} \alpha_i =\alpha_0$, and $\tilde{M}_{ij}=0$ for all $1 \leq i < j \leq \tilde{n}$.
		\item[\rm (iii)] Let $\bm{b}=((\bm{b}^1)^T,\cdots,(\bm{b}^{\tilde{n}})^T)^T=(b_1,\cdots,b_{\alpha_0})^T$, where $\bm{b}^i=(b^i_1,\cdots,b^i_{\alpha_i})^T$. Then $\tilde{M}_{ii}$ is still an $\alpha_i \times \alpha_i$ irreducible matrix for all $1 \leq i \leq \tilde{n}$ with $\sum_{i=1}^{\tilde{n}} \alpha_i =\alpha_0$, and $\tilde{M}_{ij}=0$ for all $1 \leq i < j \leq \tilde{n}$ by exchanging the order of the components of $\bm{b}^i$ such that $b^i_1< \cdots < b^i_{\alpha_i}$.
		\item[\rm (iv)] For any $1 \leq i \leq \tilde{n}$, let
		$$
		\tilde{\nu}^{i}_{j}:=
		\begin{cases}
			j,&1 \leq j \leq \alpha_0^c,\\
			\alpha_0^c +b^{i}_{j-\alpha_0^c}, & 1+ \alpha_0^c \leq j \leq \alpha_0^c +\alpha_i,
		\end{cases}
		$$
		if $\alpha_0^c>0$ and $\tilde{\nu}^{i}_{j}:=b^{i}_{j-\alpha_0^c}$, $1 \leq j \leq \alpha_i$, if $\alpha_0^c=0$,
		and  define
		$$\bm{p}_{l,i}:=((\bm{p}_{l,i}^{1})^T,\cdots,(\bm{p}_{l,i}^{\alpha_i})^T)^T,~ \bm{q}_{l,i}:=((\bm{q}_{l,i}^{1})^T,\cdots,(\bm{q}_{l,i}^{\alpha_i})^T)^T,$$
		$L^{i}:=(L_{hl}^i)_{(\alpha_0^c+\alpha_i)\times (\alpha_0^c+\alpha_i)}$, $M^{i}:=(M_{hl}^i)_{(\alpha_0^c+\alpha_i)\times (\alpha_0^c+\alpha_i)}$, and
		$\tilde{M}^{i}:=(\tilde{m}_{hl}^{i})_{\alpha_i \times \alpha_i}$ by
		$$
		\bm{p}_{l,i}^{j}=	\bm{p}_{b_{l}^{i}}^{\tilde{\nu}_{j}^{i}},
		~\bm{q}_{l,i}^{j}=	\bm{q}_{b_{l}^{i}}^{\tilde{\nu}_{j}^{i}},~ 1 \leq j \leq \alpha_0^c+ \alpha_i,~ 1\leq l \leq \alpha_i,
		$$
		$$
		L_{hl}^{i}= L_{\tilde{\nu}^{i}_h \tilde{\nu}^{i}_l},~M_{hl}^{i}= M_{\tilde{\nu}^{i}_h \tilde{\nu}^{i}_l}, ~1 \leq h, l \leq \alpha_0^c +\alpha_i.
		$$
		and 
		$$
		\tilde{m}_{hl}^{i}=\bm{p}_{h,i}^T M^i \bm{q}_{l,i}, ~1 \leq h, l \leq \alpha_i.
		$$
		Then for any   $1 \leq i \leq \tilde{n}$, we have $\tilde{M}_{ii}=\tilde{M}^{i}$, and for any $1\leq l \leq \alpha_i$,
		$$
		\bm{p}_{l,i}^T\bm{q}_{l,i}=1,~ \bm{p}_{l,i}^T\bm{q}_{h,i}=0,~ h\neq l,~
		L^{i}\bm{q}_{l,i}=\bm{0}, \text{ and }
		\bm{p}_{l,i}^T L^{i}=\bm{0}^T.
		$$
	\end{itemize}
	
\end{lemma}

\begin{proof}
	We only consider the case of $\alpha_0^c >0$, since  the case of  $\alpha_0^c =0$ can be addressed  in a similar way.
	
	(i)
	For any $1 \leq l \leq \alpha_0$, $\bm{q}_{l}^{\nu_{l}} \neq \bm{0}$ and $\bm{q}_{l}^{j} = \bm{0}$,  $j \neq \nu_{l}$, and $\bm{p}_{l}^{\nu_{l}} \neq \bm{0}$ and $\bm{p}_{l}^{j} = \bm{0}$, $j \neq \nu_{l}$ with $j>\alpha_0^c$. An easy computation yields that
	\begin{equation}
		\bm{p}_{h}^T M \bm{q}_{l}
		= \sum_{j =1}^{\alpha_0^c} (\bm{p}_{h}^{j})^T M_{j \nu_{l}} \bm{q}_{l}^{\nu_{l}} + (\bm{p}_{h}^{\nu_{h}})^T M_{\nu_{h} \nu_{l}} \bm{q}_{l}^{\nu_{l}},~ 1\leq h,l \leq \alpha_0.
	\end{equation}
	Since $M_{j\nu_{l}}$ is nonnegative for all $1 \leq j \leq \alpha_0^c$ and $M_{\nu_{h} \nu_{l}}$ is nonnegative for $h \neq l$, it follows  that $PMQ$ is cooperative.
	
	Note that exchanging the order of the components of $\bm{b}$ is equivalent to exchanging the row and column simultaneously. Thus, 
	statements (ii) and (iii) follow from \cite[Section 2.3]{berman1994nonnegative}.
	
	(iv) 
	It is easy to see that $\tilde{\nu}^{i}_{1}< \cdots<\tilde{\nu}^{i}_{\alpha_i+\alpha_0^c} $ and $\tilde{\nu}^{i}_{l+\alpha_0^c}=b_{l}^{i}+\alpha_{0}^{c}=\nu_{b^{i}_{l}}$, $\forall  1< l < \alpha_i$. Thus,
	for any $1\leq i \leq \tilde{n}$, $1 \leq l \leq \alpha_i$, $$\bm{q}_{l,i}^{\nu_l}=\bm{q}_{l,i}^{\alpha_0^c+l}=\bm{q}_{b_{l}^{i}}^{\tilde{\nu}_{\alpha_0^c+l}^{i}}=\bm{q}_{b_l^i}^{\nu_{b_l^i}} \neq \bm{0}, ~\bm{q}_{l,i}^{j}=	\bm{q}_{b_{l}^{i}}^{\tilde{\nu}_{j}^{i}}= \bm{0},  ~\forall j \neq \nu_l,$$ $$\bm{p}_{l,i}^{\nu_l}=\bm{p}_{l,i}^{\alpha_0^c+l}=\bm{p}_{b_{l}^{i}}^{\tilde{\nu}_{\alpha_0^c+l}^{i}}=\bm{p}_{b_l^i}^{\nu_{b_l^i}} \neq \bm{0},~ \bm{p}_{l,i}^{j}=	\bm{p}_{b_{l}^{i}}^{\tilde{\nu}_{j}^{i}}= \bm{0}, ~\forall j \neq \nu_l, \text{ with } j>\alpha_0^c,
	$$ 
	$$
	M_{j \nu_{l}}^i 
	=M_{j (l+ \alpha_0^c)}^i 
	=M_{\tilde{\nu}_{j}^{i}\tilde{\nu}_{l+ \alpha_0^c}^{i}}
	=M_{j\tilde{\nu}_{l+ \alpha_0^c}^{i}}
	=M_{j \nu_{b_l^i}},~ \forall 1 \leq j \leq \alpha_0^c,
	$$
	and
	$$
	M_{\nu_{h} \nu_{l}}^i
	=M_{(h+ \alpha_0^c) (l+ \alpha_0^c)}^i
	=M_{\tilde{\nu}_{h+ \alpha_0^c}^{i}\tilde{\nu}_{l+ \alpha_0^c}^{i}}
	=M_{\nu_{b_h^i} \nu_{b_l^i}}, ~ \forall 1 \leq  h \leq \alpha_i.
	$$
	Thus, for any $1\leq i \leq \tilde{n}$, $1 \leq l,h \leq \alpha_i$,
	$$
	\begin{aligned}
		\tilde{m}_{hl}^{i}
		&=\bm{p}_{h,i}^T M^i \bm{q}_{l,i}
		= \sum_{j =1}^{\alpha_0^c} (\bm{p}_{h,i}^{j})^T M_{j \nu_{l}}^i \bm{q}_{l,i}^{\nu_{l}} 
		+ (\bm{p}_{h,i}^{\nu_{h}})^T M_{\nu_{h} \nu_{l}}^i \bm{q}_{l,i}^{\nu_{l}}\\
		&=\sum_{j =1}^{\alpha_0^c} (\bm{p}_{b_h^i}^{j})^T M_{j \nu_{b_l^i}} \bm{q}_{b_l^i}^{\nu_{b_l^i}} + (\bm{p}_{b_h^i}^{\nu_{b_h^i}})^T M_{\nu_{b_h^i} \nu_{b_l^i}} \bm{q}_{b_l^i}^{\nu_{b_l^i}}.
	\end{aligned}
	$$
	We also have
	$$
	\tilde{m}_{hl}=
	\bm{p}_{b_h}^T M \bm{q}_{b_l}
	= \sum_{j =1}^{\alpha_0^c} (\bm{p}_{b_h}^{j})^T M_{j \nu_{b_l}} \bm{q}_{b_l}^{\nu_{b_l}} + (\bm{p}_{b_h}^{\nu_{b_h}})^T M_{\nu_{b_h} \nu_{b_l}} \bm{q}_{b_l}^{\nu_{b_l}}, ~1\leq h,l \leq \alpha_0.
	$$
	Since $b_{h}^{i}=b_{h}$ if $i=1$ and $b_{h}^{i}=b_{h+\sum_{j =1}^{i-1} \alpha_j}$ if $i\geq 2$,
	we obtain that $	\tilde{m}_{hl}^{i}=\tilde{m}_{hl}$ for all $1\leq h,l \leq \alpha_i$ if $i =1$ and $	\tilde{m}_{hl}^{i}=\tilde{m}_{(h+\sum_{j=1}^{i-1} \alpha_j) (l+\sum_{j=1}^{i-1} \alpha_j)}$ for all $1\leq h,l \leq \alpha_i$ if $i \geq 2$. This yields that $\tilde{M}^{i}=\tilde{M}_{ii}$, $\forall 1 \leq i \leq \tilde{n}$. Similarly, we can show that 
	$$
	\bm{p}_{l,i}^T\bm{q}_{l,i}=1,~ \bm{p}_{l,i}^T\bm{q}_{h,i}=0,~ h\neq l,~
	L^{i}\bm{q}_{l,i}=\bm{0}, \text{ and }
	\bm{p}_{l,i}^T L^{i}=\bm{0}^T
	$$
	hold for any $ 1 \leq i \leq \tilde{n}, ~1\leq l \leq \alpha_i$.
\end{proof}

Note that we define $M^i$ by choosing all indexes $\{\tilde{\nu}_j^i:  1 \leq j \leq \alpha_0^c +\alpha_i\}$ from $M$, and  define $L^{i}$, $\bm{p}_{l,i}$, $\bm{q}_{l,i}$ by using the same indexes of $L$, $\bm{p}_{l}$, $\bm{q}_{l}$, respectively. Thus,  the analysis of a reducible matrix $\tilde{M}$ can be transferred into that of its irreducible block. 

\begin{lemma}\label{lem:PQM}
	Assume that {\rm (H1)$'$} holds. Let  $\hat{P}=(\hat{p}_{lj})_{\alpha_0 \times n}$ and $\hat{Q}=(\hat{q}_{ih})_{n \times \alpha_0 }$ be two nonnegative matrices such that $\hat{P}L=0$, $L\hat{Q}=0$, and $\hat{P}\hat{Q}=I$, where $I$ is an $\alpha_0 \times \alpha_0$ identity matrix. If $M$ is an $n\times n$ matrix, then $PMQ$ is similar to $\hat{P}M\hat{Q}$.
\end{lemma}

\begin{proof}
	For any $1\leq l \leq \alpha_0$, let $\bm{\hat{p}}_l:=(\hat{P}_{l1},\cdots,\hat{P}_{ln})^T$. It is easy to see that $\bm{\hat{p}}_l^T$ is a left eigenvector of $L$ corresponding to the zero eigenvalue. Since $PQ=I$ and $\hat{P}\hat{Q}=I$, the 
	matrices $P$, $Q$, $\hat{P}$ and $\hat{Q}$ share the same rank $\alpha_0$. This implies that $\{\bm{p}_i^T: 1 \leq  i \leq \alpha_0\}$ and $\{\bm{\hat{p}}_i^T: 1 \leq  i \leq \alpha_0\}$ are  two bases of the left eigenspace of $L$ corresponding to the zero eigenvalue due to 
	Lemma \ref{lem:L_K}. Thus, there exists an 
	$\alpha_0 \times \alpha_0$ invertible matrix $A$  such that $AP=\hat{P}$.  Similarly, there exists an $\alpha_0 \times \alpha_0$ invertible  matrix $B$ such that $QB=\hat{Q}$. It then follows that 
	$AB=APQB=\hat{P}\hat{Q}=I$, and hence, $A=B^{-1}$. Therefore, $\hat{P}M\hat{Q}=APMQA^{-1}$.	
\end{proof}

In order to study the continuity of the basic reproduction ratio with respect to parameters, we next generalize the results in \cite[Theorems 2.1 and 2.2]{zhang2020asymptotic}.
Let $(\Theta,\rho_{\Theta})$ be a metric space with metric $\rho_{\Theta}$ and let $H(\mu, \theta)$ be a mapping from $\mathbb{R}_+ \times \Theta \rightarrow \mathbb{R}$. Assume that for any $\theta \in \Theta$, one of the following two properties holds:
\begin{itemize}
	\item[(P1)]  There exists a unique $\mu(\theta)>0$ such that $H(\mu(\theta),\theta)=0$, $H(\mu,\theta)<0$ for all $\mu > \mu(\theta)$ and $H(\mu,\theta)>0$ for all $\mu < \mu(\theta)$.
	\item[(P2)] $H(\mu,\theta)<0$ for all $\mu > 0$.
\end{itemize}
For convenience,  we define $\mu(\theta):=0$ in the case (P2).
Then we have the following observation.
\begin{lemma}\label{lem:R0:continuity}
	Assume that for any $\theta \in \Theta$, either {\rm(P1)} or {\rm (P2)} holds. Let $\theta_0 \in \Theta$ be given.
	If $H(\mu,\theta)$ converges to $H(\mu,\theta_0)$ as $\theta \rightarrow \theta_0$ for any $\mu>0$, then $\lim\limits_{\theta \rightarrow \theta_0}\mu(\theta)= \mu(\theta_0)$.
\end{lemma}

\begin{proof}
	We proceed according to two cases:
	
	{\it Case 1.}
	(P1) holds for $\theta_0$.
	For any $\epsilon \in (0,\mu(\theta_0))$, it follows from  (P1) that
	$$H(\mu(\theta_0)-\epsilon,\theta_0)>0, \text{ and }H(\mu(\theta_0)+\epsilon,\theta_0)<0.$$ Thus, there exists $\delta>0$ such that if $\rho_{\Theta} (\theta,\theta_0)<\delta$, then $$H(\mu(\theta_0)-\epsilon,\theta)>0,\text{ and }H(\mu(\theta_0)+\epsilon,\theta)<0.$$ Assumption (P1) implies that 
	$$
	\mu(\theta_0)-\epsilon <\mu(\theta)< \mu(\theta_0)+\epsilon 
	$$
	provided that $\rho_{\Theta}(\theta,\theta_0)<\delta$. That is, $\lim\limits_{\theta \rightarrow \theta_0}\mu(\theta)= \mu(\theta_0)$.
	
	{\it Case 2.}
	(P2) holds for $\theta_0$. It suffices to show that $\lim\limits_{\theta \rightarrow \theta_0}\mu(\theta)=0= \mu(\theta_0)$.
	For any given $\epsilon>0$, the assumption (P2) implies that
	$H(\epsilon,\theta_0)<0$.  Then there exists $\delta>0$ such that $H(\epsilon,\theta)<0$ if $\rho_{\Theta}(\theta,\theta_0)<\delta$. In view of (P1) or (P2), we conclude that $0 \leq \mu(\theta)< \epsilon$
	provided that $\rho_{\Theta}(\theta,\theta_0)<\delta$.
\end{proof}

\section{The principal eigenvalue}\label{sec:eig}
In this section, we investigate the asymptotic behavior of the principal eigenvalue for periodic cooperative patch models with large dispersal rate. We first recall some properties of time-periodic evolution families.

\begin{definition}\label{def:evol}
	A family of bounded linear operators $\varUpsilon(t,s)$, $t,s \in \mathbb{R}$ with $t\geq s$, on a Banach space $E$  is called  a $T$-periodic evolution family provided that 
	$$
	\varUpsilon(s,s)=I,\quad
	\varUpsilon(t,r)\varUpsilon(r,s)=\varUpsilon(t,s),\quad 
	\varUpsilon(t+T,s+T)=\varUpsilon(t,s),
	$$
	for all $t,s,r \in \mathbb{R}$ with $t\geq r \geq s$, and for each $e\in E$, $\varUpsilon(t,s)e$ is a continuous function of $(t,s)$ with $t \geq s$. The exponential growth bound of the evolution family $\{\varUpsilon(t,s): t\geq s\}$ is defined as 
	$$
	\omega(\varUpsilon) = \inf \{\tilde{\omega} \in \mathbb{R} : \exists M \geq 1: \forall t,s \in \mathbb{R},~t \geq s: \Vert \varUpsilon(t,s)\Vert_{E} \leq M e^{\tilde{\omega }(t-s)} \}.
	$$
\end{definition}

\begin{lemma}{\sc (\cite[Propostion A.2]{thieme2009spectral})}
	\label{lem:w_theta:equ}
	Let $\{ \varUpsilon(t,s): t\geq s \}$ be a $T$-periodic evolution family on a Banach space $E$. Then 
	$\omega(\varUpsilon)=\frac{\ln r(\varUpsilon(T,0))}{T}=\frac{\ln r(\varUpsilon(T+\tau,\tau))}{T},~ \forall \tau\in [0,T]$.
\end{lemma}

Let $M(t)=(m_{ij}(t))_{n \times n}$ be a continuous $n\times n$ matrix-valued function of $t \in \mathbb{R}$  such that
\begin{itemize}
	\item[(H4)] $M(t)=M(t+T)$ and  $M(t)$ is cooperative for all $t \in \mathbb{R}$.
\end{itemize}  

Motivated by population models in a patchy environment, we consider the following periodic ODE system
\begin{equation}\label{equ:sys:periodic}
	\frac{\mathrm{d} \bm{v}}{\mathrm{d} t}=d L \bm{v} + M(t) \bm{v},
\end{equation}
and the associated eigenvalue problem:
\begin{equation}\label{equ:eig:periodic}
	\frac{\mathrm{d} \bm{u}}{\mathrm{d} t}=d L \bm{u} + M(t) \bm{u} -\lambda \bm{u}.
\end{equation}

\begin{definition}\label{def:principal}
	$\lambda^*$ is called the principal eigenvalue of \eqref{equ:eig:periodic} if it is a real eigenvalue with a nonnegative eigenfunction and the real parts
	of all other eigenvalues are not greater than $\lambda^*$.
\end{definition}

For any $d \geq 0$, according to \cite[Section 7.3]{krasnoselskij1964positive} or \cite[Chapter 5]{pazy1983semigroups}, system \eqref{equ:sys:periodic} admits a unique evolution family $\{\mathbb{O}_d(t,s): t \geq s \}$ on $\mathbb{R}^n$ with $\mathbb{O}_d(t,s) \bm{\phi}=\bm{u}(t,s;\bm{\phi})$, $\forall t\geq s$ and $\bm{\phi} \in \mathbb{R}^n$, where $\bm{u}(t,s;\bm{\phi})$ is the unique solution at time $t$ of \eqref{equ:sys:periodic} with initial data $\bm{\phi}$ at time $s$. In view of \cite[Theorem 7.17]{krasnoselskij1964positive} and Lemma \ref{lem:w_theta:equ}, we have the following result (see also \cite[Theorem 2.7]{liang2017principal}).
\begin{theorem}\label{thm:existence}
	Assume that {\rm (H1)} and {\rm (H4)} hold. Then the eigenvalue problem \eqref{equ:eig:periodic} admits the principal eigenvalue $\lambda_d^*=  \omega (\mathbb{O}_d)= \frac{\ln r(\mathbb{O}_d(T,0))}{T}$ for all 
	$d \geq 0$.
\end{theorem}

The subsequent result is a consequence of the standard comparison arguments.
\begin{lemma}\label{lem:comparison}
	Assume that {\rm (H4)} holds.
	Let $\hat{M}(t)=(\hat{m}_{ij}(t))_{n \times n}$ be a continuous $n\times n$ matrix-valued function of $t \in \mathbb{R}$ with $\hat{M}(t)=\hat{M}(t+T)$ such that $\hat{M}(t)$ is cooperative for all $t \in \mathbb{R}$. Let $\lambda^*_M$ and $\lambda^*_{\hat{M}}$ be the principal eigenvalue of $\frac{\mathrm{d} \bm{u}}{\mathrm{d} t}= M(t) \bm{u} -\lambda \bm{u}$ and $\frac{\mathrm{d} \bm{u}}{\mathrm{d} t}= \hat{M}(t) \bm{u} -\lambda \bm{u}$, respectively. If $m_{ij}(t) \geq \hat{m}_{ij}(t)$, $\forall 1 \leq i,j \leq n$, $t \in \mathbb{R}$, then $\lambda^*_{M} \geq \lambda^*_{\hat{M}}$. Further, if $M(t)$ can be split into 
	$$
	M(t)=
	\left(
	\begin{array}{cc}
		M_{11}(t)&  M_{12}(t) \\ 
		M_{2 1}(t)& M_{2 2}(t)\\
	\end{array} 
	\right),
	$$
	then $\lambda_M^* \geq \lambda^*_{M_{11}}$, where $\lambda^*_{M_{11}}$ is the principal eigenvalue of $\frac{\mathrm{d} \bm{u}}{\mathrm{d} t}= M_{11}(t) \bm{u} -\lambda \bm{u}$.
\end{lemma}
From now on, we let $\lambda_{d}^{*}$ be the principal eigenvalue of \eqref{equ:eig:periodic} and $\bm{u}_d=(u_{d,1},\cdots,u_{d,n})^T$ be an nonnegative eigenvector corresponding to $\lambda_{d}^{*}$ for any given $d \geq 0$.
For convenience, we normalize $\bm{u}_d$ by
$\max_{1 \leq i \leq n} \max_{t \in \mathbb{R}} u_{d,i}(t)=1$. 
\begin{lemma}\label{lem:bounded:1}
	Assume that {\rm (H1)} and {\rm (H4)} hold. Then there exists real number $C>0$, independent of $d$,  such that $ \vert \lambda_{d}^{*} \vert \leq C$.
\end{lemma}

\begin{proof}
	Let  $\overline{M}:= \max_{1 \leq i,j \leq n} \{\max_{ t\in \mathbb{R}} m_{ij}(t)\}$ and $\underline{M}:= \min_{1 \leq i,j \leq n}\{ \min_{ t\in \mathbb{R}} m_{ij}(t)\}$, and define two $n \times n$ matrices $M^1:=(m_{ij}^1)_{n\times n} $ by $m_{ij}^{1}=\overline{M}$ and $M^2:=\mathrm{diag}(\underline{M},\cdots\underline{M})$. Let $\overline{\lambda}$ and $\underline{\lambda}$ be the principal eigenvalue of 
	$dL+ M^1$ and $dL+ M^2$, respectively. We use $\bm{e}=(1,\cdots,1)^T$ to denote an $n$-dimensional vector. By the Perron-Frobenius theorem (see, e.g., \cite[Theorem 4.3.1]{smith2008monotone}), it then follows from $\bm{e}^T(dL+ M^1)=n\overline{M}\bm{e}^T$ and $\bm{e}^T(dL+ M^2)=\underline{M}\bm{e}^T$ that $\overline{\lambda}= n \overline{M}$ and $\underline{\lambda}= \underline{M}$. In view of Lemma \ref{lem:comparison}, we have $\underline{\lambda}\leq \lambda \leq \overline{\lambda}$.
\end{proof}

For any $d>0$, we define 
$$\tilde{u}^l_d(t):= \bm{p}_l^T \bm{u}_d(t),~ \forall 1 \leq l \leq \alpha_0,~\tilde{\bm{u}}_d(t):= \sum_{l=1}^{\alpha_0} \tilde{u}^l_d(t) \bm{q}_l
\text{ and }
\hat{\bm{u}}_d(t):=\bm{u}_d(t) - \tilde{\bm{u}}_d(t).
$$

\begin{lemma}\label{lem:hat:ud}
	Assume that {\rm (H1)} and {\rm (H4)} hold. Then $\sup_{t \in \mathbb{R}}\Vert \hat{\bm{u}}_d(t) \Vert_{\mathbb{R}^n} \rightarrow 0$ as $d \rightarrow +\infty$.
\end{lemma}

\begin{proof}
	Our arguments are motivated by \cite{hale1986large,hale1987varying,hale1989shadow,hutson2001evolution,zhang2020asymptotic}.
	Define
	$$
	X_1:={\rm Span}\{\bm{q}_l\}_{1 \leq l \leq \alpha_0}
	\text{ and }
	X_2:=\{\bm{q} \in \mathbb{R}^n: \bm{p}_l^T \bm{q}=0,~ 1 \leq l \leq \alpha_0 \}.
	$$
	It then follows that
	$$
	\mathbb{R}^n= X_1 \oplus X_2.
	$$
	Let $S_d(t)$ be the semigroup generated by $dL$, that is, $S_d(t)=e^{dLt}$. It is easy to see that $S_d(t) X_1 \subseteq X_1$ and $S_d(t) X_2 \subseteq X_2$. According to \cite[Theorem  7.3]{daners1992abstract}, we then have
	$$
	\Vert S_d(t) \bm{\phi} \Vert_{\mathbb{R}^n} \leq C_1 e^{-\gamma_0 d t} \Vert \bm{\phi} \Vert_{\mathbb{R}^n}, ~ \forall \bm{\phi} \in X_2
	$$
	for some $\gamma_0>0$ and $C_1>0$, independent of $t$ and $d$. We multiply  \eqref{equ:eig:periodic} from left by $\bm{p}_l^T$ to obtain
	$$
	\frac{\mathrm{d} }{\mathrm{d} t} \tilde{u}_d^{l}= \bm{p}_l^T M(t) \bm{u}_d- \lambda_{d}^{*}\tilde{u}_d^{l}, ~ \forall 1 \leq l \leq \alpha_0,
	$$
	and then multiply the above equation by $\bm{q}_l$ to get
	$$
	\frac{\mathrm{d} }{\mathrm{d} t} (\tilde{u}_d^{l}\bm{q}_l)= [\bm{p}_l^T M(t) \bm{u}_d]  \bm{q}_l- \lambda_{d}^{*}\tilde{u}_d^{l}\bm{q}_l, ~ \forall 1 \leq l \leq \alpha_0.
	$$
	Adding them together yields 
	\begin{equation}\label{equ:bmu:d}
		\frac{\mathrm{d} }{\mathrm{d} t} \tilde{\bm{u}}_d= \sum_{l=1}^{\alpha_0}[\bm{p}_l^T M(t) \bm{u}_d]  \bm{q}_l- \lambda_{d}^{*}\tilde{\bm{u}}_d.
	\end{equation}
	Subtracting \eqref{equ:bmu:d} from \eqref{equ:eig:periodic}, we have
	$$
	\frac{\mathrm{d} }{\mathrm{d} t} \hat{\bm{u}}_d=
	d L \hat{\bm{u}}_d + M(t) \bm{u}_d-
	\sum_{l=1}^{\alpha_0}[\bm{p}_l^T M(t) \bm{u}_d]  \bm{q}_l- \lambda_{d}^{*}\hat{\bm{u}}_d.
	$$ 
	Clearly, for any $1 \leq l \leq \alpha_0$,
	$$
	\bm{p}_l^T L \hat{\bm{u}}_d =0,~
	\bm{p}_l^T \hat{\bm{u}}_d=\bm{p}_l^T (\bm{u}_d-\tilde{\bm{u}}_d)= \tilde{u}_d^l- \tilde{u}_d^l(\bm{p}_l^T \bm{q}_l)= 0,
	$$
	and
	$$
	\bm{p}_l^T\left(M(t) \bm{u}_d-
	\sum_{l=1}^{\alpha_0}[\bm{p}_l^T M(t) \bm{u}_d]  \bm{q}_l\right)= 
	\bm{p}_l^TM(t) \bm{u}_d- 
	\bm{p}_l^T M(t) \bm{u}_d  (\bm{p}_l^T \bm{q}_l)=0.
	$$
	That is,
	$L \hat{\bm{u}}_d \in X_2$, $\hat{\bm{u}}_d \in X_2$, and 
	$M(t) \bm{u}_d-
	\sum_{l=1}^{\alpha_0}(\bm{p}_l^T M(t) \bm{u}_d) \bm{q}_l \in X_2$.
	By Lemma \ref{lem:bounded:1}, there exists a $C_2>0$, independent of $d$ and $t$, such that
	$$
	\left \Vert M(t) \bm{u}_d-
	\sum_{l=1}^{\alpha_0}[\bm{p}_l^T M(t) \bm{u}_d]  \bm{q}_l \right\Vert_{\mathbb{R}^n} \leq C_2
	\text{ and }
	\vert \lambda_{d}^{*} \vert \leq C_2.
	$$
	In view of  the constant-variation formula, we obtain
	$$
	\hat{\bm{u}}_d(t) = S_d(t) \hat{\bm{u}}_d(0) + 
	\int_{0}^{t} S_d(t-s) \left\{
	M(s) \bm{u}_d(s) -
	\sum_{l=1}^{\alpha_0}\left[\bm{p}_l^T M(s) \bm{u}_d(s) \right]  \bm{q}_l- \lambda_{d}^{*}\hat{\bm{u}}_d(s)
	\right\} \mathrm{d} s,
	$$
	for all $t \geq 0$.
	An easy computation gives rise to
	$$
	\Vert \hat{\bm{u}}_d (t) \Vert_{\mathbb{R}^n} 
	\leq C_1 e^{-\gamma_0 d t} \Vert \hat{\bm{u}}_d (0) \Vert_{\mathbb{R}^n} +C_2 \int_{0}^{t} e^{-\gamma_0 d (t-s)} \mathrm{d} s+ C_2 \int_{0}^{t} e^{-\gamma_0 d (t-s)} \Vert \hat{\bm{u}}_d (s) \Vert_{\mathbb{R}^n}   \mathrm{d} s.
	$$
	Choose $\gamma_1 \in (0,\gamma_0)$, and define 
	$\zeta_d(t):= e^{\gamma_1 d t } \Vert \hat{\bm{u}}_d(t) \Vert_{\mathbb{R}^n}$, 
	$\overline{\zeta}_d(t):= \sup \{\zeta_d(s): 0 \leq s \leq t \}$ and 
	$$
	C_3:= \int_{0}^{\infty} e^{-s[1 -\gamma_1(\gamma_0)^{-1}]} \mathrm{d} s.
	$$
	It then follows that
	$$
	\zeta_d (t) \leq C_1 e^{-(\gamma_0-\gamma_1)dt} \zeta_d (0) + C_2 C_3(\gamma_0 d)^{-1} e^{\gamma_1 d t} + C_2 C_3 (\gamma_0 d)^{-1} \overline{\zeta}_d (t), ~ t \geq 0,
	$$
	and hence,
	$$
	\overline{\zeta}_d (t) \leq C_1 e^{-(\gamma_0-\gamma_1)dt} \zeta_d (0) + C_2 C_3(\gamma_0 d)^{-1} e^{\gamma_1 d t} + C_2 C_3 (\gamma_0 d)^{-1} \overline{\zeta}_d (t), ~ t \geq 0.
	$$
	For any $d >0$, let $\xi (d):=C_2 C_3 (\gamma_0 d)^{-1}$. Notice that $\xi (d) \rightarrow 0$ as $d \rightarrow +\infty$. From now on, we assume that $d$ is large enough such that $\xi(d) < \frac{1}{2}$, which implies that $(1 -\xi(d))^{-1} \leq 2$. This leads to 
	$$
	\begin{aligned}
		\overline{\zeta}_d (t) 
		&\leq (1 -\xi(d))^{-1}[C_1 e^{-(\gamma_0-\gamma_1)dt} \zeta_d (0) + C_2 C_3(\gamma_0 d)^{-1} e^{\gamma_1 d t}]\\
		&\leq 2[C_1 e^{-(\gamma_0-\gamma_1)dt} \zeta_d (0) + C_2 C_3(\gamma_0 d)^{-1} e^{\gamma_1 d t}],
	\end{aligned}
	$$
	and hence,
	$$
	\Vert \hat{\bm{u}}_d(t) \Vert_{\mathbb{R}^n}
	\leq e^{-\gamma_1 d t} \overline{\zeta}_d (t)
	\leq 2[C_1 e^{-(\gamma_0-\gamma_1)dt} \zeta_d (0)  e^{-\gamma_1 d t} + C_2 C_3(\gamma_0 d)^{-1}].
	$$
	Letting $t \rightarrow +\infty$, we obtain
	$$
	\limsup\limits_{t \rightarrow +\infty} \Vert \hat{\bm{u}}_d(t) \Vert_{\mathbb{R}^n} \leq 2 C_2 C_3 (\gamma_0 d)^{-1}.
	$$
	Since $\hat{\bm{u}}(t)$ is periodic in $t \in \mathbb{R}$, it  follows that
	$$
	\Vert \hat{\bm{u}}_d(t) \Vert_{\mathbb{R}^n} \leq 2 C_2 C_3 (\gamma_0 d)^{-1}.
	$$
	This yields the desired conclusion.
\end{proof}

Define $\tilde{M}(t)=(\tilde{m}_{hl}(t))_{\alpha_0 \times \alpha_0}$ by
$\tilde{m}_{hl}(t)=\bm{p}_h^T M(t) \bm{q}_l$, that is, $\tilde{M}(t)=P M(t) Q$.
Let $\{\tilde{O}(t,s): t \geq s \}$ be the evolution family on $\mathbb{R}^{\alpha_0}$ of 
$$
\frac{\mathrm{d} \bm{v}}{\mathrm{d} t}= \tilde{M}(t) \bm{v},
$$
and let $\tilde{\lambda}^{*}$ be the principal eigenvalue of 
$$
\frac{\mathrm{d} \bm{u}}{\mathrm{d} t}= \tilde{M}(t) \bm{u} -\lambda \bm{u}.
$$
It is easy to see that $\omega(\tilde{O})= \tilde{\lambda}^{*}$ due to Theorem \ref{thm:existence}. The following result indicates that $\tilde{\lambda}^{*}$ is independent of the choice of $P$ and $Q$.
\begin{lemma}\label{lem:PQO}
	Assume that {\rm (H1)$'$} holds. Let  $\widehat{P}=(\widehat{p}_{lj})_{\alpha_0 \times n}$ and $\widehat{Q}=(\widehat{q}_{ih})_{n \times \alpha_0 }$ be two nonnegative matrices such that $\widehat{P}L=0$, $L\widehat{Q}=0$ and $\widehat{P}\widehat{Q}=I$, where $I$ is an $\alpha_0 \times \alpha_0$ identity matrix. Let $\widehat{M}(t):=\widehat{P} M(t) \widehat{Q}$ and $\{\widehat{O}(t,s): t \geq s \}$ be the evolution family on $\mathbb{R}^{\alpha_0}$ of 
	$\frac{\mathrm{d} \bm{v}}{\mathrm{d} t}= \widehat{M}(t) \bm{v}$. Then $\widehat{O}(T,0)$ is similar to $\tilde{O}(T,0)$. Moreover, $\omega(\widehat{O})=\omega(\tilde{O})$.
\end{lemma}

\begin{proof}
	According to Lemma \ref{lem:PQM}, there exists an $\alpha_0 \times \alpha_0$ invertible matrix $A$ such that $AP=\widehat{P}$ and $QA^{-1}=\widehat{Q}$. By a change of variable $\bm{w}=A^{-1}\bm{v}$,  we then transfer $\frac{\mathrm{d} \bm{v}}{\mathrm{d} t}= \widehat{M}(t) \bm{v}$ into $\frac{\mathrm{d} \bm{w}}{\mathrm{d} t} = \tilde{M}(t) \bm{w}$.  Thus, we have  $\widehat{O}(T,0)=A[\tilde{O}(T,0)]A^{-1}$.
\end{proof}

\begin{lemma}\label{lem:eig:irreducible}
	Assume that {\rm (H1)} and {\rm (H4)} hold.
	If $\tilde{O}(T,0)$ is irreducible, then
	$\lim\limits_{d \rightarrow +\infty} \lambda_{d}^{*}= \tilde{\lambda}^{*}$. 
\end{lemma}

\begin{proof}
	For any $1 \leq h \leq \alpha_0$, we multiply  \eqref{equ:eig:periodic} from left by $\bm{p}_h^T$ to obtain
	\begin{equation}\label{equ:udh}
		\frac{\mathrm{d} }{\mathrm{d} t}\tilde{u}_d^h=  \bm{p}_h^TM(t) \bm{u}_d -\lambda_{d}^{*} \tilde{u}_d^h.
	\end{equation}
	Then there exists $C_1>0$ such that
	$$
	\left\vert \frac{\mathrm{d} }{\mathrm{d} t}\tilde{u}_d^h \right\vert \leq C_1, ~\forall1 \leq  h \leq \alpha_0.
	$$
	By the Ascoli–Arzel\`{a} theorem (see, e.g., \cite[Theorem I.28]{reed1980methods}), it follows that there exists a sequence $d_m \rightarrow +\infty$ such that $\lambda_{d_m} \rightarrow \lambda_{\infty}$ and $\vert \tilde{u}^h_{d_m} (t) - \tilde{u}^h_{\infty} (t)\vert \rightarrow 0$ uniformly for $t \in \mathbb{R}$, $1 \leq h \leq \alpha_0$, as $m \rightarrow +\infty$, for some $\lambda_{\infty}$ and $\tilde{u}^h_{\infty} \in C(\mathbb{R},\mathbb{R}_+)$ with $\tilde{u}^h_{\infty}(t+T)=\tilde{u}^h_{\infty}(t)$, $\forall t \in \mathbb{R},~ 1 \leq h \leq \alpha_0$.
	We integrate  \eqref{equ:udh} from $0$ to $t$ to obtain
	$$
	\tilde{u}_d^h(t)- \tilde{u}_d^h(0)
	=\int_{0}^{t} [\bm{p}_h^TM(s) \tilde{\bm{u}}_d(s)+ \bm{p}_h^TM(s) \hat{\bm{u}}_d(s) -\lambda_{d}^{*} \tilde{u}_d^h(s) ]\mathrm{d} s.
	$$
	By Lemma \ref{lem:hat:ud}, letting $d_m \rightarrow +\infty$, for any $1 \leq h \leq \alpha_0$, we have
	$$
	\tilde{u}_{\infty}^h(t)- \tilde{u}_{\infty}^h(0)
	=\int_{0}^{t} 
	\left[\bm{p}_h^TM(s) \left(\sum_{l=1}^{\alpha_0} \tilde{u}_{\infty}^l(s)  \bm{q}_l \right)-\lambda_{\infty} \tilde{u}_{\infty}^h(s) \right]\mathrm{d} s,
	$$ 
	and hence,
	$$
	\frac{\mathrm{d} }{\mathrm{d} t} \tilde{u}_{\infty}^h(t)
	=\sum_{l=1}^{\alpha_0}[\bm{p}_h^TM(t) \bm{q}_l ] \tilde{u}_{\infty}^l(t) -\lambda_{\infty} \tilde{u}_{\infty}^h(t).
	$$ 
	Letting $\bm{\phi}=(\tilde{u}_{\infty}^1 (0),\cdots, \tilde{u}_{\infty}^{\alpha_0}(0))^T$, we see that
	$$
	\bm{\phi}=e^{-\lambda_{\infty} T}\tilde{O}(T,0) \bm{\phi}.
	$$
	With the irreducibility of $\tilde{O}(T,0)$,
	the Perron-Frobenius theorem (see, e.g., \cite[Theorem 4.3.1]{smith2008monotone}) then leads to $\lambda_{\infty}=\tilde{\lambda}^{*}$.
\end{proof}

To remove the irreducibility condition on $\tilde{O}(T,0)$  in Lemma \ref{lem:eig:irreducible}, below we prove the same conclusion as in Lemma \ref{lem:bounded:1} under weaker conditions.

\begin{lemma}
	Assume that {\rm (H1)$'$} and {\rm (H4)} hold. Then there exists some $C>0$ such that $ \vert \lambda_{d}^{*} \vert \leq C$.
\end{lemma}

\begin{proof}
	We proceed according to two cases:
	
	{\it Case 1.} $\Lambda_0^c= \emptyset$. The proof is motivated by the arguments for Lemma \ref{lem:bounded:1}.
	Define $$\overline{M}:= \max_{1 \leq i,j \leq n} \{\max_{ t\in \mathbb{R}} m_{ij}(t)\}, \text{ and }\underline{M}:= \min_{1 \leq i,j \leq n} \{\min_{ t\in \mathbb{R}} m_{ij}(t)\}.$$
	For any $1 \leq i \leq \alpha_0$, choose $\bm{p}^i\gg0$ such that $(\bm{p}^i)^T L_{ii}= \bm{0}^T$. Let $\bm{p}=((\bm{p}^1)^T,\cdots,(\bm{p}^{\alpha_0})^T)^T=(p_{1},\cdots,p_{n})^T$. Thus, $\bm{p}^T L=\bm{0}^T$.
	Without loss of generality, we assume that $\min_{1 \leq j \leq n}p_j=  1$. Define two $n \times n$ matrices $M^1:=(m_{ij}^1)_{n\times n} $ by $m_{ij}^{1}=\overline{M} p_j$, $\forall 1 \leq i,j \leq n$ and $M^2:=\mathrm{diag}(\underline{M},\cdots,\underline{M} )$. Let $\overline{\lambda}$ and $\underline{\lambda}$ be the principal eigenvalue of 
	$dL+ M^1$ and $dL+ M^2$, respectively. In view of $\bm{p}^T(dL+ M^1)= (\sum_{j =1}^{n} p_j)\overline{M}\bm{p}^T$ and $\bm{p}^T(dL+ M^2)=\underline{M}\bm{p}^T$, the Perron-Frobenius theorem (see, e.g., \cite[Theorem 4.3.1]{smith2008monotone}) implies that $\overline{\lambda}= (\sum_{j =1}^{n} p_j)\overline{M}$ and $\underline{\lambda}= \underline{M}$.  By Lemma \ref{lem:comparison}, it easily follows that $\underline{\lambda}\leq \lambda \leq \overline{\lambda}$.
	
	{\it Case 2.} $\Lambda_0^c \neq \emptyset$.
	Without loss of generality, in view of (H1)$'$, we assume that $\Lambda_0^c=\{1,\cdots,\alpha_0^c\}$ and $\Lambda_0=\{\alpha_0^c+1,\cdots,\alpha\}$, and still write $\nu_l:=\alpha_0^c+l$, $ 1\leq l \leq \alpha_0$, as Lemma \ref{lem:L_K}. Let us first prove that $\lambda_{d}^{*}$ has a lower bound independent of $d$. We split the matrix-valued function $M(t)$ into a block form as follows
	$$
	M(t)=
	\left(
	\begin{array}{ccc}
		M_{11}(t)& \cdots & M_{1\alpha}(t) \\ 
		\vdots& \ddots& \vdots\\
		M_{\alpha 1}(t)& \cdots & M_{\alpha \alpha}(t)\\
	\end{array} 
	\right),
	$$
	where $M_{hl}$ is an $n_h \times n_l $ matrix for $1 \leq h,l  \leq \alpha$. Define a matrix $\hat{L}= \mathrm{diag}(L_{\nu_1\nu_1},\cdots,L_{\alpha\alpha})$ and a matrix-valued function $\hat{M}(t)$ by
	$$
	\hat{M}(t)=\left(
	\begin{array}{ccc}
		M_{\nu_1\nu_1}(t)& \cdots & M_{\nu_1 \alpha}(t) \\ 
		\vdots& \ddots& \vdots\\
		M_{\alpha\nu_1}(t)& \cdots & M_{\alpha \alpha}(t)\\
	\end{array} 
	\right).
	$$  
	Let $\hat{\lambda}_{d}^{*}$ be the principal eigenvalue of 
	$$
	\frac{\mathrm{d} \bm{u}}{\mathrm{d} t}=d \hat{L} \bm{u} + \hat{M}(t) \bm{u} -\lambda \bm{u}.
	$$
	Since $s(L_{ll})=0$ for all $\nu_1 \leq l \leq \alpha$, and $\hat{M}(t)$ is cooperative for any $t \in \mathbb{R}$, it then follows from Lemma \ref{lem:comparison} that $\hat{\lambda}_{d}^{*} \leq \lambda_{d}^{*}$. By the proof of Case 1, $\hat{\lambda}_{d}^{*}$ has a lower bound independent of $d$, so does $\lambda_{d}^{*}$.
	
	We next show that $\lambda_{d}^{*}$ has an upper bound independent of $d$.
	Define a matrix $\overline{L}$ by
	$$
	\overline{L}=
	\left(
	\begin{array}{ccc}
		\overline{L}_{11}& \cdots & \overline{L}_{1\alpha} \\ 
		\vdots& \ddots& \vdots\\
		\overline{L}_{\alpha 1}& \cdots & \overline{L}_{\alpha \alpha}\\
	\end{array} 
	\right),
	$$
	where $\overline{L}_{hl}=L_{hl}$, for $ 1 \leq h \leq \alpha_0^c$, $1 \leq l \leq \alpha $ and $ \nu_1 \leq h,l \leq \alpha$ and $\overline{L}_{hl}=L_{hl} + \bm{e}_{n_h} \bm{e}_{n_l}^T$ for $ \nu_1 \leq h \leq \alpha$, $1 \leq l \leq \alpha_0^c $. Here $\bm{e}_{n_h} =(1,\cdots,1)^T$ is an $n_h$-dimensional vector. For any $1 \leq l \leq \alpha_0$,  choose $\bm{p}^{\nu_l} \gg 0$ such that $(\bm{p}^{\nu_l})^T L_{\nu_l\nu_l}=\bm{0}^T$. Since all elements of $\overline{L}_{hl}$ are positive for $ \nu_1 \leq h \leq \alpha$, $1 \leq l \leq \alpha_0^c $ and $L_{ll}$ is irreducible for $1 \leq l \leq \alpha_0^c$,  by the arguments similar to those for \eqref{equ:p:Lhl}, there exist $\bm{p}^i \gg 0$, $1 \leq i \leq \alpha_0^c$ such that
	$
	\sum_{i =1}^\alpha (\bm{p}^i)^T \overline{L}_{ih}=\bm{0}^T,~ 1 \leq  h  \leq \alpha_0^c. 
	$
	Define 
	$
	\bm{p}:=((\bm{p}^1)^T,\cdots,(\bm{p}^\alpha)^T)^T
	$, where $\bm{p}^i$ is an $n_i$-dimensional vector. By repeating the arguments for the upper bound in Case 1, we obtain the desired conclusion.
\end{proof}
\begin{remark}\label{rem:eig:irreducible}
	Assume that {\rm (H1)$'$} and {\rm (H4)} hold. If $\tilde{O}(T,0)$ is irreducible, then
	$\lim\limits_{d \rightarrow +\infty} \lambda_{d}^{*}= \tilde{\lambda}^{*}$.
\end{remark}

The following result provides a powerful tool to analyze the matrix $\tilde{O}(T,0)$ in the case where  it is reducible.
\begin{lemma}\label{lem:widehat:O}
	For any $\alpha_0$-dimensional vector $\bm{b}=(b_1,\cdots,b_{\alpha_0})^T$
	with $1 \leq b_i \leq \alpha_0$ and $b_i \neq b_j$ if $i \neq j$, define  $\widehat{M} (t)=(\widehat{m}_{hl}(t))_{\alpha_0 \times \alpha_0}$ by
	$\widehat{m}_{hl}(t)=\bm{p}_{b_h}^T M(t) \bm{q}_{b_l}$.  Let $\{\widehat{O}(t,s): t\geq s \}$ be the evolution family of $\frac{\mathrm{d} \bm{v}}{\mathrm{d} t}= \widehat{M}(t) \bm{v}$ on $\mathbb{R}^{\alpha_0}$. Then the matrix $\tilde{O}(T,0)$ is similar to the matrix $\widehat{O}(T,0)$. If, in addition, $\tilde{O}(T,0)$ is reducible, then $\widehat{O}(T,0)$ is a block lower triangular matrix after choosing a suitable $\bm{b}$.
\end{lemma}
The following two results are straightforward consequence of \cite[Lemmas 3.5 and 3.7]{zhang2020asymptotic}.
\begin{lemma}\label{lem:split}
	Write $A:=\tilde{O}(T,0)=(a_{ij})_{\alpha_0 \times \alpha_0}$ and let
	$$
	A=
	\left(
	\begin{matrix}
		A_{11} &  \cdots & A_{1 \tilde{n}}\\
		\vdots & \ddots &\vdots\\
		A_{\tilde{n} 1} & \cdots & A_{\tilde{n} \tilde{n}}
	\end{matrix}
	\right),\quad 
	\text{and} \,  \, \,
	\tilde{M}(t)=
	\left(
	\begin{matrix}
		\tilde{M}_{11}(t) &  \cdots & \tilde{M}_{1 \tilde{n}}(t)\\
		
		\vdots & \ddots &\vdots\\
		\tilde{M}_{\tilde{n} 1} (t)&  \cdots & \tilde{M}_{\tilde{n} \tilde{n}}(t)
	\end{matrix}
	\right),
	$$
	where $A_{ii}$ is an $\alpha_i \times \alpha_i$ matrix with $\sum_{i=1}^{\tilde{n}} \alpha_i =\alpha_0$, and $\tilde{M}_{ii}(t)$ is an $\alpha_i \times \alpha_i$ matrix-valued function of $t \in \mathbb{R}$. If $A_{ij}$ are zero matrices for all $1 \leq i < j \leq \tilde{n}$, then so are $\tilde{M}_{ij}(t)$ for any $ t\in \mathbb{R}$. Moreover, let $\tilde{\lambda}_{i}^{*}$ be the principal eigenvalue of $\frac{\mathrm{d} \bm{u}}{\mathrm{d} t}=\tilde{M}_{ii} (t) \bm{u} - \lambda \bm{u},~t>0$, then $\tilde{\lambda}_{i}^{*}= \frac{\ln r(A_{ii})}{T}$ and $\tilde{\lambda}^{*}=\max_{1 \leq i \leq \tilde{n}} \tilde{\lambda}_{i}^{*}$.
\end{lemma}
\begin{lemma}\label{lem:conti_limits}
	Let $g$ be a continuous function on $(a,b)$ and write $g_+=\limsup_{x \rightarrow b} g(x)$ and $g_-=\liminf_{x \rightarrow b} g(x)$. Then for any $c \in [g_-,g_+]$, there exists a sequence $x_k \rightarrow b$ as $k \rightarrow \infty$ with $x_k \in (a,b)$ such that $\lim\limits_{k \rightarrow \infty} g(x_k)=c$.
\end{lemma}

Now we are in a position to prove the main result of this section.
\begin{theorem}\label{thm:eig}
	Assume that {\rm (H1)} and {\rm (H4)} hold.
	Then the following statements are valid:
	\begin{itemize}
		\item[\rm (i)] $\lim\limits_{d \rightarrow 0^+} \lambda_{d}^{*}= \lambda_{0}^{*}$ and $\lim\limits_{d \rightarrow \hat{d}} \lambda_{d}^{*}= \lambda_{\hat{d}}^{*}$ for any $\hat{d}>0$.
		\item[\rm (ii)] $\lim\limits_{d \rightarrow +\infty} \lambda_{d}^{*}= \tilde{\lambda}^{*}$. 
	\end{itemize}
\end{theorem}
\begin{proof}
	(i) We only prove that $\lim\limits_{d \rightarrow 0^+} \lambda_{d}^{*}= \lambda_{0}^{*}$, since $\lim\limits_{d \rightarrow \hat{d}} \lambda_{d}^{*}= \lambda_{\hat{d}}^{*}$ can be derived for any $\hat{d}>0$ in a similar way. Since solutions of \eqref{equ:sys:periodic}  depend continuously upon parameters (see, e.g., \cite[Section I.3]{hale1969ordinary}), it follows that $\mathbb{O}_d(T,0)$ converges to $\mathbb{O}_0(T,0)$ in the matrix norm as $d \rightarrow 0^+$.
	For the definition of the matrix norm, we refer to \cite[Section II.2]{steward1990matrices}. Therefore, the desired statement (i) follows from the perturbation theory of matrix (see, e.g., \cite{kato1976perturbation,steward1990matrices}).
	
	(ii) Our proof is motivated by the arguments for \cite[Theorem 3.3]{zhang2020asymptotic}. Since 
	the conclusion has been proved in the case where $\tilde{O}(T,0)$ is irreducible in Lemma \ref{lem:eig:irreducible}, we only need to consider the case where that $\tilde{O}(T,0)$ is reducible.
	We proceed in three steps.
	
	{\it Step 1.} $\lambda_{\infty}:=\lim_{d\rightarrow +\infty} \lambda_{d}^{*}$ exists. According to Lemma \ref{lem:bounded:1},
	both $\lambda_+:=\limsup_{d\rightarrow +\infty} \lambda_{d}^{*}$ and $\lambda_-:=\liminf_{d \rightarrow +\infty} \lambda_{d}^{*}$ exist, and $C_1 \leq \lambda_-,\lambda_+ \leq C_2$ for some $C_1$ and $C_2$. It suffices to prove that $\lambda_-=\lambda_+$. Suppose that
	$\lambda_-< \lambda_+$, for any $\hat{\lambda} \in [\lambda_-, \lambda_+]$,
	by repeating the arguments in the proof of Lemma \ref{lem:eig:irreducible}, there exists a positive vector $\bm{\phi}$ such that
	$$
	\bm{\phi}=e^{-\hat{\lambda}T}\tilde{O}(T,0) \bm{\phi}.
	$$
	This implies that $e^{\hat{\lambda}T}$ is an eigenvalue of $\tilde{O}(T,0)$ for any $\hat{\lambda} \in [\lambda_-, \lambda_+]$, which is impossible.
	
	{\it Step 2.} $\lambda_{\infty} \leq \tilde{\lambda}^{*}$. For any given $\epsilon>0$, let $M^{\epsilon}=(m_{ij}^{\epsilon})_{n \times n}$ and $\tilde{M}^{\epsilon}=(\tilde{m}_{hl}^{\epsilon})_{\alpha_0 \times \alpha_0}$ be two continuous matrix-valued functions of $t \in \mathbb{R}$ with $m_{ij}^{\epsilon}(t)=m_{ij}(t)+\epsilon$, $\forall t \in \mathbb{R}$ and $\tilde{m}_{hl}^{\epsilon}(t)=\bm{p}_h^T M^{\epsilon}(t) \bm{q}_l$, $\forall t \in \mathbb{R}$. Let $\tilde{\lambda}^{*}(\epsilon)$ be the principal eigenvalue of the eigenvalue problem 
	$$
	\frac{\mathrm{d} \bm{u}}{\mathrm{d} t}=\tilde{M}^{\epsilon}(t) \bm{u} - \lambda \bm{u}, ~t>0.
	$$ 
	Let $\lambda_{d}^{*}(\epsilon)$ be the principal eigenvalue of the eigenvalue problem \eqref{equ:eig:periodic} with $M$ replaced by $M^{\epsilon}$. Clearly, $\lambda_{d}^{*}(\epsilon) \geq \lambda_{d}^{*}$ for all $\epsilon>0$ due to Lemma \ref{lem:comparison}. It then follows that
	$$\tilde{\lambda}^{*}(\epsilon)
	=\lim\limits_{d \rightarrow +\infty} \lambda_{d}^{*}(\epsilon) 
	\geq \lim\limits_{d \rightarrow +\infty} \lambda_{d}^{*}
	=\lambda_{\infty}, \forall \epsilon >0.$$
	Since $\tilde{\lambda}^{*}(\epsilon) \rightarrow \tilde{\lambda}^{*}$ as $\epsilon \rightarrow 0^+$, we conclude that
	$\lambda_{\infty}\leq \lim\limits_{\epsilon \rightarrow 0^+}\tilde{\lambda}^{*}(\epsilon)= \tilde{\lambda}^{*}$.
	
	{\it Step 3.} $\lambda_{\infty} \geq \tilde{\lambda}^{*}$. We only consider the case of $\alpha_0^c>0$, since the case of $\alpha_0^c=0$ can be addressed in a similar way.
	Without loss of generality, by Lemma \ref{lem:L_K}, we assume that $\Lambda_0^c=\{1,\cdots,\alpha_0^c\}$ and $\Lambda_0=\{\alpha_0^c+1,\cdots,\alpha\}$.
	Based on Lemma \ref{lem:widehat:O}, we can redefine  $\tilde{M} (t)=(\tilde{m}_{hl}(t))_{\alpha_0 \times \alpha_0}$ by
	$\tilde{m}_{hl}(t)=\bm{p}_{b_h}^T M(t) \bm{q}_{b_l}$ for a specified $\alpha_0$-dimensional vector $\bm{b}$ such that
	\begin{itemize}
		\item[(1)] $\bm{b}=(b_1,\cdots,b_{\alpha_0})^T$ with $1 \leq b_i \leq \alpha_0$ and $b_i \neq b_j$ if $i \neq j$.
		\item[(2)] The matrix $A:=\tilde{O}(T,0)$ can be split into
		$$
		A=
		\left(
		\begin{matrix}
			A_{11} &  \cdots & A_{1 \tilde{n}}\\
			\vdots & \ddots &\vdots\\
			A_{\tilde{n} 1} & \cdots & A_{\tilde{n} \tilde{n}}
		\end{matrix}
		\right),
		$$
		where $A_{ij}$ is an $\alpha_i \times \alpha_j$ matrix with $\sum_{i=1}^{\tilde{n}} \alpha_i =\alpha_0$, $A_{ij}=0$ for all $1 \leq i < j \leq \tilde{n}$ and $A_{ii}$ is irreducible for all $1 \leq i \leq \tilde{n}$.
		\item[(3)] $\bm{b}=((\bm{b}^1)^T,\cdots,(\bm{b}^{\tilde{n}})^T)^T$, where  $\bm{b}^i=(b^i_1,\cdots,b^i_{\alpha_i})^T$ with $b^i_1< \cdots < b^i_{\alpha_i}$ for all $1 \leq i  \leq \tilde{n}$.
	\end{itemize}  
	Here (3) is achievable because both (1) and (2) are still valid by exchanging the components of $\bm{b}^i$ due to Lemma \ref{lem:L_K_M}.
	By Lemma \ref{lem:split}, $\tilde{M}(t)$ can be split into
	$$
	\tilde{M}(t)=
	\left(
	\begin{matrix}
		\tilde{M}_{11}(t) &  \cdots & \tilde{M}_{1 \tilde{n}}(t)\\
		
		\vdots & \ddots &\vdots\\
		\tilde{M}_{\tilde{n} 1} (t)&  \cdots & \tilde{M}_{\tilde{n} \tilde{n}}(t)
	\end{matrix}
	\right),
	$$
	where $\tilde{M}_{ij}$ is an $\alpha_i \times \alpha_j$ matrix-valued function with $\tilde{M}_{ij}(t)=0$ for all $1 \leq i < j \leq \tilde{n}$. Let $\tilde{\lambda}_{i}^{*}$ be the principal eigenvalue of the eigenvalue problem
	$$\frac{\mathrm{d} \bm{u}}{\mathrm{d} t}=\tilde{M}_{ii}(t) \bm{u} - \lambda \bm{u},~t>0.$$ 
	Thus, Lemma \ref{lem:split} yields that $\tilde{\lambda}_{i}^{*}= \frac{\ln r(A_{ii})}{T}$ and $\tilde{\lambda}^{*}=\max_{1 \leq i \leq \tilde{n}} \tilde{\lambda}_{i}^{*}$.
	For any $1 \leq i \leq \tilde{n}$ and $t \in \mathbb{R}$, define $L^{i}$, $M^{i}(t)$ and $\tilde{M}^{i}(t)$ as $L^{i}$, $M^{i}$ and $\tilde{M}^{i}$ in Lemma \ref{lem:L_K_M}. For any $1 \leq i \leq \tilde{n}$, $1 \leq l \leq \alpha_i$, choose $\bm{p}_{l,i}$ and $\bm{q}_{l,i}$ in the same way as in Lemma \ref{lem:L_K_M}.
	It then follows from Lemma \ref{lem:L_K_M} that for any $1 \leq i \leq \tilde{n}$,  $\tilde{M}_{ii}(t)=\tilde{M}^{i}(t)$, $\forall t \in \mathbb{R}$, and for any $1\leq l \leq \alpha_i$
	$$
	\bm{p}_{l,i}^T\bm{q}_{l,i}=1,~ \bm{p}_{l,i}^T\bm{q}_{h,i}=0,~ h\neq l,~
	L^{i}\bm{q}_{l,i}=\bm{0}, \text{ and }
	\bm{p}_{l,i}^T L^{i}=\bm{0}^T.
	$$
	Let $\lambda_{d,i}^{*}$ be the principal eigenvalue of 
	\begin{equation}\label{equ:eig:periodic:i}
		\frac{\mathrm{d} \bm{u}}{\mathrm{d} t}=d L^i \bm{u} + M^i(t) \bm{u} -\lambda \bm{u}.
	\end{equation}
	Since $A_{ii}$ is irreducible and (H1)$'$ holds for $L^i$, we conclude that $\lambda_{d,i}^{*} \rightarrow \tilde{\lambda}_{i}^{*}$ as $d \rightarrow +\infty$ due to Lemma \ref{lem:eig:irreducible} and Remark \ref{rem:eig:irreducible}. It then  follows from Lemma \ref{lem:comparison} that $\lambda_{d,i}^{*}  \leq \lambda_{d}^{*}$.
	Notice that
	$$
	\tilde{\lambda}_{i}^{*}=\lim_{d \rightarrow +\infty} \lambda_{d,i}^{*} \leq \lim_{d \rightarrow +\infty} \lambda_{d}^{*} =\lambda_{\infty}, ~1 \leq i \leq \tilde{n}.
	$$
	Thus, Lemma \ref{lem:split} implies that $\tilde{\lambda}^{*}=\max_{1 \leq i \leq \tilde{n}} \tilde{\lambda}_{i}^{*} \leq \lambda_{\infty}$. 
\end{proof}

\begin{remark}\label{rem:eig:reducible}
	Assume that {\rm (H1)$'$} and {\rm (H4)} hold. Then $\lim\limits_{d \rightarrow 0^+} \lambda_{d}^{*}= \lambda_{0}^{*}$, $\lim\limits_{d \rightarrow \hat{d}} \lambda_{d}^{*}= \lambda_{\hat{d}}^{*}$ for any $\hat{d}>0$, and 
	$\lim\limits_{d \rightarrow +\infty} \lambda_{d}^{*}= \tilde{\lambda}^{*}$.
\end{remark}

\section{The basic reproduction ratio}\label{sec:R0}

In this section, we study the continuity of the basic reproduction ratio with respect to parameters and investigate its asymptotic behavior  as the dispersal rates go to infinity for a periodic patch model. 

In order to discuss the continuity of the basic reproduction ratio with respect to parameters, we introduce a metric space $(\mathcal{X}, \rho_{\mathcal{X}})$ with metric $\rho_{\mathcal{X}}$.
For any given $\chi \in \mathcal{X}$, let  $F_{\chi}(t)=(f_{ij,\chi}(t))_{n\times n}$
and $V_{\chi}(t)=(v_{ij,\chi}(t))_{n\times n}$  be two continuous $n\times n$ matrix-valued functions of $t \in \mathbb{R}$ such that 

\begin{enumerate}
	\item[(H2)$'$]  $F_{\chi}(t+T)=F_{\chi}(t)$, $V_{\chi}(t+T)=V_{\chi}(t)$, 
	$F_{\chi}(t)$ is nonnegative, and $-V_{\chi}(t)$ is cooperative
	for all $\chi \in \mathcal{X}$ and $t \in \mathbb{R} $.
\end{enumerate}

Let  $\tilde{F}_{\chi}(t)=(\tilde{f}_{hl,\chi}(t))_{\alpha_0 \times \alpha_0}= PF_{\chi}(t)Q$ and $\tilde{V}_{\chi}(t)=(\tilde{v}_{hl,\chi}(t))_{\alpha_0 \times \alpha_0}= PV_{\chi}(t)Q$ for all $t \in \mathbb{R}$.
For any $d \geq 0 $, let $\{\Phi_{d,\chi}(t,s): t \geq s \}$ be the evolution family on $\mathbb{R}^{n}$ of  $\frac{\partial \bm{v}}{\partial t}= d L \bm{v} - V_{\chi}(t) \bm{v}$, $t\geq s$.
We use $\{\tilde{\Phi}_{\chi} (t,s): t \geq s\}$ to denote the evolution family on $\mathbb{R}^{\alpha_0}$ of 
$ \frac{\partial \bm{v}}{\partial t}= - \tilde{V}_{\chi}(t) \bm{v}$, $t\geq s$.   
We further assume that
\begin{enumerate}
	\item[(H3)$'$]  $\omega (\Phi_{d,\chi})<0$ for all $d \geq 0$ and $\omega (\tilde{\Phi}_{\chi})<0$.
\end{enumerate}

It is easy to see that (H2)$'$ and (H3)$'$ are generalizations of (H2) and (H3), respectively.
For any $\mu >0$ and $d \geq 0$, let $\{\mathbb{U}_{d,\chi}^{\mu}(t,s): t \geq s \} $ be the evolution family on $\mathbb{R}^{n}$ of 
\begin{equation}\label{equ:U_kappa}
	\frac{\partial \bm{v}}{\partial t}= d L \bm{v} - V_{\chi}(t) \bm{v}+ \frac{1}{\mu} F_{\chi} (t) \bm{v},  ~t \geq s,
\end{equation}
and let $\{\tilde{\mathbb{U}}_{\chi}^{\mu}(t,s): t \geq s \} $ be the evolution family on $\mathbb{R}^{\alpha_0}$ of 
\begin{equation}\label{equ:U_infty}
	\frac{\partial \bm{v}}{\partial t}= - \tilde{V}_{\chi}(t) \bm{v} +\frac{1}{\mu} \tilde{F}_{\chi} (t) \bm{v},~t \geq s.
\end{equation}

Define
$$
\mathbb{X}:=\{ \bm{u} \in C(\mathbb{R},\mathbb{R}^n): \bm{u}(t)=\bm{u}(t+T), ~t \in \mathbb{R} \},
$$ 
$$
\tilde{\mathbb{X}}:=\{ \bm{u} \in C(\mathbb{R},\mathbb{R}^{\alpha_0}): \bm{u}(t)=\bm{u}(t+T), ~t \in \mathbb{R} \},
$$
$$
\mathbb{X}_+:=\{ \bm{u} \in C(\mathbb{R},\mathbb{R}_+^n): \bm{u}(t)=\bm{u}(t+T),~t \in \mathbb{R}\},
$$
and
$$
\tilde{\mathbb{X}}_+:=\{ \bm{u} \in C(\mathbb{R},\mathbb{R}_+^{\alpha_0}): \bm{u}(t)=\bm{u}(t+T),~t \in \mathbb{R}\}
$$
with the maximum norm $\Vert \bm{u} \Vert_{\mathbb{X}} = \max_{1 \leq i \leq n}\max_{0 \leq t \leq T} \vert u_i (t) \vert$ for $\bm{u}=(u_1,u_2,\cdots,u_n)^T$ and $\Vert \bm{u} \Vert_{\tilde{\mathbb{X}}} = \max_{1 \leq i \leq \alpha_0}\max_{0 \leq t \leq T} \vert u_i (t) \vert$ for $\bm{u}=(u_1,u_2,\cdots,u_{\alpha_0})^T$. Then $(\mathbb{X},\mathbb{X}_+)$ and $(\tilde{\mathbb{X}},\tilde{\mathbb{X}}_+)$ are two ordered Banach spaces.

For any $d \geq 0$, we define a bounded linear positive operator $\mathbb{Q}_{d,\chi}: \mathbb{X} \rightarrow \mathbb{X}$ by
$$
[\mathbb{Q}_{d,\chi} \bm{u}](t): = \int_{0}^{+\infty} \Phi_{d,\chi}(t,t-s)F_{\chi}(t-s)\bm{u}(t-s) \mathrm{d} s, ~ t \in \mathbb{R},~ \bm{u} \in \mathbb{X},
$$
and $\mathcal{R}_0(d,\chi):=r(\mathbb{Q}_{d,\chi})$.
Define $\tilde{\mathbb{Q}}_{\chi} : \tilde{\mathbb{X}} \rightarrow \tilde{\mathbb{X}}$ by
$$
[\tilde{\mathbb{Q}}_{\chi} \bm{u}] (t):= \int_{0}^{+\infty} \tilde{\Phi}_{\chi}(t,t-s)\tilde{F}_{\chi}(t-s) \bm{u}(t-s) \mathrm{d} s, ~ t \in \mathbb{R},~ \bm{u} \in \tilde{\mathbb{X}},
$$
and $\tilde{\mathcal{R}}_0(\chi):=r(\tilde{\mathbb{Q}}_{\chi})$.
The subsequent result is a straightforward consequence of \cite[Theorems 2.1 and 2.2]{wang2008threshold}.

\begin{lemma}\label{lem:R0:equivalent}
	Assume that {\rm (H1)}, {\rm (H2)$'$} and {\rm (H3)$'$} hold. Then the following statements are valid for any $\mu >0$ and $\chi \in \mathcal{X}$:
	\begin{itemize}
		\item[\rm (i)]  For any $d \geq 0$, $\mathcal{R}_0(d,\chi) -\mu$ has the same sign as
		$\omega(\mathbb{U}_{d,\chi}^{\mu})$.
		\item[\rm (ii)]  $\tilde{\mathcal{R}}_0(\chi) -\mu$ has the same sign as
		$\omega(\tilde{\mathbb{U}}_{\chi}^{\mu})$.
	\end{itemize}
\end{lemma}
Now we are ready to prove the main result of this section.
\begin{theorem}\label{thm:R0}
	Assume that {\rm (H1)}, {\rm (H2)$'$} and {\rm (H3)$'$} hold, and there exists $\chi_0 \in \mathcal{X}$
	such that  
	$V_{\chi}$ and $F_{\chi}$ converge to $V_{\chi_0}$ and $F_{\chi_0}$ in the matrix norm as $\chi \rightarrow \chi_0$, respectively. Then the following statements are vaild:
	\begin{itemize}
		\item[\rm (i)] $\lim\limits_{d\rightarrow 0^+,\chi \rightarrow \chi_0} \mathcal{R}_0(d,\chi)= \mathcal{R}_0(0,\chi_0)$, and
		$\lim\limits_{d\rightarrow \hat{d},\chi \rightarrow \chi_0} \mathcal{R}_0(d,\chi)= \mathcal{R}_0(\hat{d},\chi_0)$ for any $\hat{d}>0$. 
		\item[\rm (ii)]
		$\lim\limits_{d\rightarrow +\infty,\chi \rightarrow \chi_0} \mathcal{R}_0(d,\chi)= \tilde{R}_{0}(\chi_0).$
	\end{itemize}
\end{theorem}

\begin{proof}
	(i) Without loss of generality, we only prove
	$$\lim\limits_{d\rightarrow 0^+,\chi \rightarrow \chi_0} \mathcal{R}_0(d,\chi)= \mathcal{R}_0(0,\chi_0).$$  In this case, we choose  $\Theta=\mathbb{R}_+ \times \mathcal{X}$ and $\theta_0:=(0,\chi_0)\in \Theta$. Define $$H(\mu,\theta):=\omega(\mathbb{U}_{d,\chi}^{\mu}),~ \forall
	\mu \in \mathbb{R}_+,~ \theta:=(d,\chi)\in \Theta.$$
	According to Lemma \ref{lem:R0:equivalent}, for any $d \geq 0$ and $\chi \in \mathcal{X}$, $H(\mathcal{R}_0(d,\chi),(d,\chi))=0$, $H(\mu,(d,\chi))<0$ for all $\mu > \mathcal{R}_0(d,\chi)$ and $H(\mu,(d,\chi))>0$ for all $\mu < \mathcal{R}_0(d,\chi)$.
	By Lemma \ref{lem:R0:continuity}, it suffices to show that for any $ \mu >0$, $\lim\limits_{d \rightarrow 0^+,\chi \rightarrow \chi_0}\omega(\mathbb{U}_{d,\chi}^{\mu}) = \omega(\mathbb{U}^{\mu}_{0,\chi_0}) $, which can be derived by the arguments similar to those in the Claim of \cite[Theorem 4.1]{zhang2020asymptotic} and Theorem \ref{thm:eig}.
	
	(ii) Now we choose  $\Theta=({\rm Int}(\mathbb{R}_+) \times \mathcal{X})\cup \{\theta_0\}$, where $\theta_0:=(0,\chi_0)$. Let $\kappa= \frac{1}{d}$ and 
	$\theta:=(\kappa,\chi)\in \Theta$ with $d >0$. Define 
	$$H(\mu,\theta):=\omega(\mathbb{U}_{d,\chi}^{\mu}),~ \forall
	\mu \in \mathbb{R}_+,~ \theta:=(\kappa,\chi)\in \Theta \setminus \{ \theta_0 \},$$ 
	and 
	$$
	H(\mu,\theta_0)= \omega(\tilde{\mathbb{U}}_{\chi}^{\mu}),~ \forall
	\mu \in \mathbb{R}_+.
	$$
	According to Lemma \ref{lem:R0:equivalent}, for any $\kappa > 0$ and $\chi \in \mathcal{X}$, $H(\mathcal{R}_0(\kappa,\chi),(\kappa,\chi))=0$, $H(\mu,(\kappa,\chi))<0$ for all $\mu > \mathcal{R}_0(\kappa,\chi)$ and $H(\mu,(\kappa,\chi))>0$ for all $\mu < \mathcal{R}_0(\kappa,\chi)$.
	By Lemma \ref{lem:R0:continuity}, it suffices to show that for any $ \mu >0$, $\lim\limits_{d \rightarrow +\infty,\chi \rightarrow \chi_0}\omega(\mathbb{U}_{d,\chi}^{\mu}) = \omega(\tilde{\mathbb{U}}^{\mu}_{\chi_0}) $, which can be derived by the arguments similar to those in the Claim of \cite[Theorem 4.1]{zhang2020asymptotic} and Theorem \ref{thm:eig}.
\end{proof}

To finish this section, we apply Theorem \ref{thm:R0} to a periodic Ross-Macdonald  model in a patch environment. According to \cite{gao2014periodic}, we consider the following $T$-periodic patch system:
\begin{subequations}\label{eqn:model:RM}\small
	\begin{align}
		&\frac{\mathrm{d} H_i}{\mathrm{d} t} = d \sum_{j=1}^{m}   l_{ij}^H H_j(t), &1 \leq i \leq m,\label{eqn:model:RM:a}\\
		&\frac{\mathrm{d} V_i}{\mathrm{d} t} =  \epsilon_i(t) - \mu_i(t) V_i(t),&1 \leq i \leq m,\label{eqn:model:RM:b}\\
		&\frac{\mathrm{d} h_i}{\mathrm{d} t} = \sigma_{1i} \beta_i (t)\frac{H_i(t) -h_i(t)}{H_i(t)} v_i(t) - \gamma_i h_i(t) +  d\sum_{j=1}^{m}  l_{ij}^H h_j(t),&1 \leq i \leq m,\label{eqn:model:RM:c}\\
		&\frac{\mathrm{d} v_i}{\mathrm{d} t} = \sigma_{2i} \beta_i(t) \frac{h_i(t)}{H_i(t)} (V_i(t)-v_i(t)) - \mu_i(t) v_i(t),&1 \leq i \leq m\label{eqn:model:RM:d}.
	\end{align}
\end{subequations}
Here $H_i(t)$ and $V_i(t)$ are the total populations of humans and mosquitoes  in patch $i$ at time $t$, respectively; $h_i(t)$ and $v_i(t)$ denote the numbers of infectious humans and mosquitoes in patch $i$ at time $t$, respectively; $\epsilon_i(t)>0$ is the recruitment rate of mosquitoes in patch $i$ at time $t$; $\mu_i(t)>0$ is the mortality rate of mosquitoes in patch $i$ at time $t$; $\sigma_{1i}>0$ $(\sigma_{2i}>0)$ are transmission probability from infectious mosquitoes (humans)
to susceptible humans (mosquitoes) in patch $i$ at time $t$;
$\beta_i(t)>0$ is the mosquito biting rate in patch $i$ at time $t$;
$\gamma^{-1}_i>0$ is the human infectious period; $l_{ij}^H$ is the degree of human migration from patch $j$ to patch $i$ for $i \neq j$; $l_{ii}^H$ is the degree of human migration from patch $i$ to all other patches; $d$ is the migration coefficients.
We assume that there is no death or birth during travel, so the emigration rate of humans in patch $i$ satisfies
$\sum_{j=1}^{m} l_{ji}^H=0,~ \forall 1 \leq i \leq m
$; the functions $\epsilon_i$, $\mu_i$ and $\beta_i$ are $T$-periodic and continuous on $\mathbb{R}$.

We further assume that the total populations of human $N^H=\sum_{j=1}^{m} H_j(0)>0$. By \cite[Lemma 3.1 ]{gao2014periodic}, it then follows that \eqref{eqn:model:RM:a} admits a globally asymptotically stable equilibrium $\bm{H}^*=(H_{1}^*,\cdots,H_{n}^*)^T$, which is independent of $d>0$ and $t \in \mathbb{R}$; and that \eqref{eqn:model:RM:b} admits a  globally asymptotically stable $T$-periodic solution $\bm{V}^*(t)=(V_{1}^*(t),\cdots,V_{n}^*(t))^T$, which is independent of $d>0$. Moreover, $\sum_{j=1}^{m} H_{j}^*=\sum_{j=1}^{m} H_j(0)$.
We linearize system \eqref{eqn:model:RM} at the disease-free periodic solution 
$$
(H_{1}^*,\cdots,H_{n}^*,V_{1}^*,\cdots,V_{n}^*,0,\cdots,0,0,\cdots,0)^T
$$
to obtain
\begin{equation}
	\begin{cases}
		\frac{\mathrm{d} h_i}{\mathrm{d} t} = \sigma_{1i} \beta_i (t) v_i(t) - \gamma_i h_i(t) + d\sum_{j=1}^{m}    l_{ij}^{H} h_j(t),&1 \leq i \leq m,\\
		\frac{\mathrm{d} v_i}{\mathrm{d} t} = \sigma_{2i} \beta_i(t) \frac{V_i^*(t)}{H_{i}^*}h_i(t)  - \mu_i(t) v_i(t),&1 \leq i \leq m.
	\end{cases}
\end{equation}
We next choose $n=2m$, and hence, 
$
\mathbb{X}:=\{ \bm{u} \in C(\mathbb{R},\mathbb{R}^{2m}): \bm{u}(t)=\bm{u}(t+T), ~t \in \mathbb{R} \}.
$ Let 
$$
V_{11}(t):=(\delta_{ij}\gamma_i)_{m \times m},~
V_{22}(t):=(\delta_{ij}\mu_i(t))_{m \times m},\text{ and }
V(t):=
\left(
\begin{array}{cc}
	V_{11}(t)& 0 \\ 
	0& V_{22}(t)
\end{array} 
\right).
$$   
Define 
$$
F_{12}(t):=(\delta_{ij}\sigma_{1i} \beta_i (t))_{m \times m},~ F_{21}(t):=\left(\delta_{ij}\sigma_{2i} \beta_i(t) \frac{V_i^*(t)}{H_{i}^*}\right)_{m \times m},
$$
and
$$
F(t):=
\left(
\begin{array}{cc}
	0& F_{12} (t) \\ 
	F_{21}(t)& 0
\end{array} 
\right).
$$
Let $L=(l_{ij})_{2m \times 2m}$ be a cooperative matrix with zero column sum defined by $l_{ij}=l_{ij}^H$, $1\leq i,j \leq m$ and $l_{ij}=0$ if otherwise.
For any $d \geq 0 $, let $\{\Phi_{d}(t,s): t \geq s \}$ be the evolution family on $\mathbb{R}^{2m}$ of  
$$
\frac{\partial \bm{v}}{\partial t}= d L \bm{v} - V(t) \bm{v}, ~t\geq s,
$$
and define a bounded linear positive operator $\mathbb{Q}_{d}: \mathbb{X} \rightarrow \mathbb{X}$ by
$$
[\mathbb{Q}_{d} \bm{u}](t): = \int_{0}^{+\infty} \Phi_{d}(t,t-s)F(t-s)\bm{u}(t-s) \mathrm{d} s, ~ t \in \mathbb{R},~ \bm{u} \in \mathbb{X},
$$
and	 $\mathcal{R}_0 (d):=r(\mathbb{Q}_{d})$. According to Theorem \ref{thm:R0}, we see that $\mathcal{R}_0(d)$ is continuous with respect to $d\in(0,+\infty)$. Indeed,  Theorem \ref{thm:R0} shows that the basic reproduction ratio is continuous with respect to all parameters in the model.

Now we turn to the limiting profile of $\mathcal{R}_0(d)$ as $d \rightarrow +\infty$. 
Let $\bm{q}=(q_1,\cdots,q_m)^T$ be a strongly positive vector such that $L^H\bm{q}=\bm{0}$ and $\sum_{i =1}^{m} q_i =1$, where $L^H=(l_{ij}^H)_{m \times m}$. Notice that $L$ is a reducible matrix. Since $L^H$ is irreducible, we have $\alpha_0=m+1$ due to Lemma \ref{lem:L_K}, and hence, $
\tilde{\mathbb{X}}:=\{ \bm{u} \in C(\mathbb{R},\mathbb{R}^{m+1}): \bm{u}(t)=\bm{u}(t+T), ~t \in \mathbb{R} \}
$. Moreover, $P=(p_{hj})_{\alpha_0\times 2m}$ and $Q=(q_{il})_{2m \times\alpha_0}$ can be defined by
$p_{1j}=1$, $1\leq j \leq m$, $p_{h(h+m-1)}=1$, $2 \leq h \leq \alpha_0$, $q_{i1}=q_{i}$, $1\leq i \leq m$, $q_{(l+m-1)l}=1$, $2\leq l \leq \alpha_0$, $p_{hj}=0$ and $q_{il}=0$ if otherwise.

For any $t \in \mathbb{R}$, define $\tilde{V}(t):=PV(t)Q$, $\tilde{F}(t):= PF(t)Q$.
Let $\{\tilde{\Phi}(t,s): t \geq s \}$ be the evolution family on $\mathbb{R}^{\alpha_0}$ of  $$\frac{\partial \bm{v}}{\partial t}=  - \tilde{V}(t) \bm{v}, ~t\geq s,$$
and define a bounded linear positive operator $\tilde{\mathbb{Q}}: \tilde{\mathbb{X}} \rightarrow  \tilde{\mathbb{X}}$ by
$$
[\tilde{\mathbb{Q}} \bm{u}](t): = \int_{0}^{+\infty} \tilde{\Phi}(t,t-s)\tilde{F}(t-s)\bm{u}(t-s) \mathrm{d} s, ~ t \in \mathbb{R},~ \bm{u} \in \tilde{\mathbb{X}},
$$
and $\tilde{\mathcal{R}}_0:=r(\tilde{\mathbb{Q}})$. It then follows from Theorem \ref{thm:C} that $\mathcal{R}_0(d) \rightarrow \tilde{\mathcal{R}}_0$ as $d \rightarrow +\infty$.

At last we numerically compute $\mathcal{R}_0$ by using the algorithm developed in  \cite{wang2008threshold,liang2019basic}. The baseline parameters are $m=2$, $T=365$, $N^H=500$, $\sigma_{1i}=0.2$, $\sigma_{2i}=0.3$, $\gamma_i=0.02$, $\mu=0.1$, as derived from \cite{gao2014periodic}, $\epsilon_1(t)=12.5-5 \cos(\frac{2\pi t}{T})-5 \cos(\frac{4\pi t}{T})$, $\epsilon_2(t)=12.5- 5\cos(\frac{2\pi t}{T})$, $\beta_i(t)=0.028  \epsilon_i(t)$, $l_{12}^H=1$, and $l_{21}^H=1$. Our numerical result shows that the basic reproduction ratios on patches 1 and 2 are $\mathcal{R}_0^{(1)}=1.5340$ and  $\mathcal{R}_0^{(2)}=1.4478$, respectively. From
Figure \ref{fig1} we observe that the dependence of $\mathcal{R}_0$ with respect to $d$ may be very complicated: $\mathcal{R}_0$ is decreasing when $d$ is small enough and large enough, while it is increasing on an interval. Moreover, $\mathcal{R}_0(d) \rightarrow \max(\mathcal{R}_0^{(1)},\mathcal{R}_0^{(2)})$ as $d\rightarrow 0$, and $\mathcal{R}_0 \rightarrow \tilde{\mathcal{R}}_0=1.5028$ as $d\rightarrow +\infty$. 
For the corresponding time-averaged autonomous system, we found that 
its basic reproduction number is $\bar{\mathcal{R}}_0=1.3555$, which is independent of $d$. This suggests  that the use of a time-averaged autonomous model may underestimate the disease severity in some transmission settings.

\begin{figure}[htbp]
	\centering
	\includegraphics[width=4in]{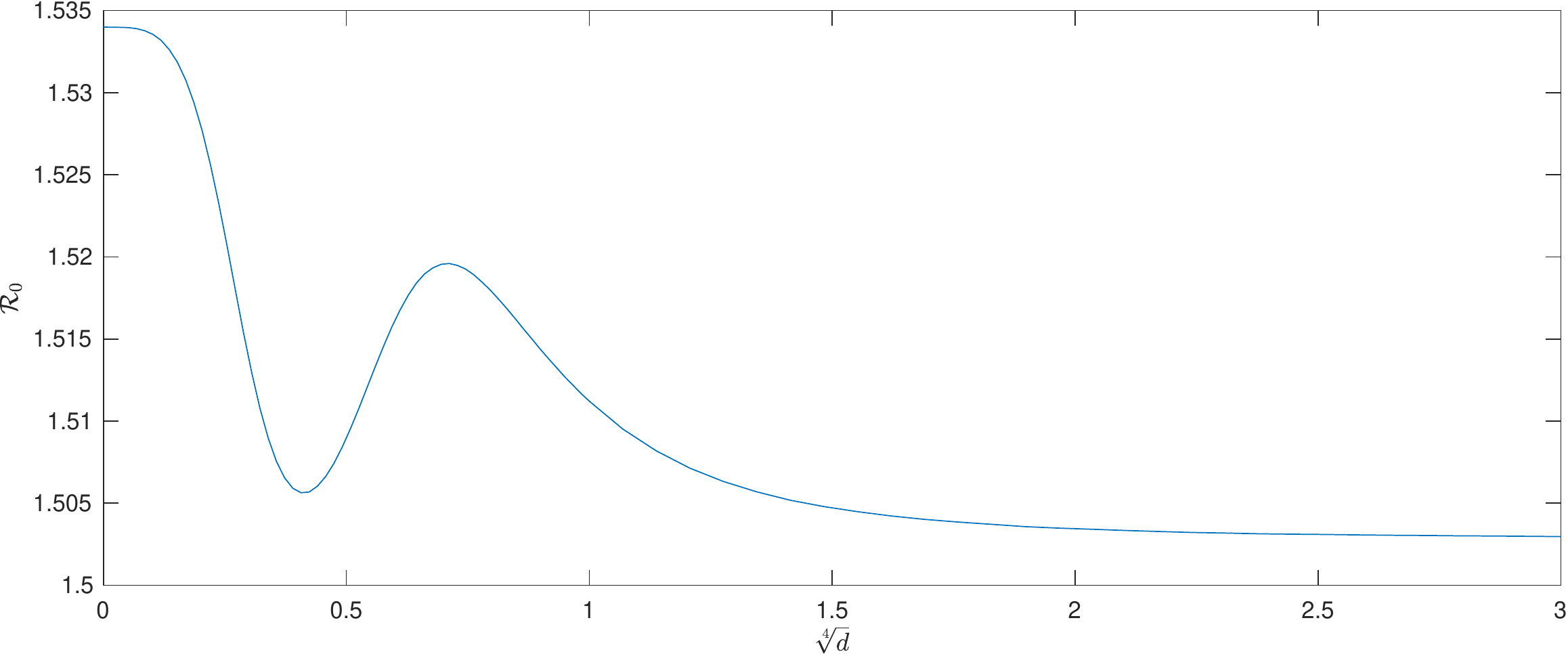}
	\caption{{\small \it  $\mathcal{R}_0$ initially decreases then increases and finally decreases with respect to $d$. }}\label{fig1}
\end{figure}

\

{\noindent \bf Acknowledgements.}
L. Zhang's research is supported in part by the National Natural Science Foundation of China (11901138) and  the Natural Science Foundation of Shandong Province (ZR2019QA006), and X.-Q. Zhao's research is supported in part by the NSERC of Canada. We are grateful to the anonymous referees for their careful reading and valuable comments, which led to an improvement of our original manuscript.


\end{document}